\documentclass[11pt]{amsart}

\usepackage[USenglish,english]{babel}
\usepackage{amsfonts}
\usepackage{amsmath}
\usepackage{amssymb}
\usepackage{amsthm}
\usepackage[T1]{fontenc}
\usepackage[latin1]{inputenc}
\usepackage{enumerate}
\usepackage{graphicx}
\usepackage[all]{xy}
\usepackage{accents}
\usepackage{mathtools}
\usepackage{mathrsfs}
\usepackage{comment}
\usepackage{accents}
\usepackage{tikz}
\usetikzlibrary{matrix,arrows,decorations.pathmorphing}
\usepackage{tikz-cd}
\usepackage{afterpage}
\usepackage{bbm}
\usepackage{dsfont}
\usepackage{stmaryrd}
\usepackage{hyperref}
\usepackage{mleftright}
\usepackage{subfig}
\usepackage{faktor}  
\usepackage[foot]{amsaddr}

\theoremstyle{plain}
\newtheorem{thm}{Theorem}[section]
\newtheorem{lem}[thm]{Lemma}
\newtheorem{prop}[thm]{Proposition}
\newtheorem{cor}[thm]{Corollary}
\newtheorem*{thm*}{Theorem}
\newtheorem*{thmA}{Theorem~A}
\newtheorem*{thmB}{Theorem~B}

\newtheorem*{propC}{Proposition~C}

\newtheorem*{thmD}{Theorem~D}
\newtheorem*{thmE}{Theorem~E}
\newtheorem*{thmF}{Theorem~F}

\newtheorem*{cor*}{Corollary}

\theoremstyle{definition}
\newtheorem{defn}[thm]{Definition}
\newtheorem*{dfn*}{Definition}

\newtheorem{rmk}[thm]{Remark}

\renewcommand{\o}{\circ}
\newcommand{\wt}{\widetilde}

\newcommand{\R}{\mathbb{R}}
\newcommand{\Z}{\mathbb{Z}}
\newcommand{\N}{\mathbb{N}}

\newcommand{\s}{\sigma}

\newcommand{\ra}{\rightarrow}
\newcommand{\cu}{\subseteq}

\newcommand{\G}{\Gamma}

\newcommand{\mbb}{\mathbb}
\newcommand{\mc}{\mathcal}
\newcommand{\mf}{\mathfrak}
\newcommand{\x}{\times}

\newcommand{\Om}{\Omega}

\newcommand{\acts}{\curvearrowright}

\newcommand{\wh}{\widehat}
\newcommand{\mscr}{\mathscr}

\begin{document}

\title{The Tits alternative for finite rank median spaces}
\author{Elia Fioravanti}
\address{Andrew Wiles Building, University of Oxford, Radcliffe Observatory Quarter, Woodstock Rd, Oxford OX2 6GG, United Kingdom}
\email{elia.fioravanti@maths.ox.ac.uk}

\begin{abstract}
We prove a version of the Tits alternative for groups acting on complete, finite rank median spaces. This shows that group actions on finite rank median spaces are much more restricted than actions on general median spaces. Along the way, we extend to median spaces the Caprace-Sageev machinery \cite{CS} and part of Hagen's theory of unidirectional boundary sets \cite{Hagen}.
\end{abstract}

%%%
% abstract for arXiv:
%We prove a version of the Tits alternative for groups acting on complete, finite rank median spaces. This shows that group actions on finite rank median spaces are much more restricted than actions on general median spaces. Along the way, we extend to median spaces the Caprace-Sageev machinery and part of Hagen's theory of unidirectional boundary sets.

% should invert arrows in definition of Hagen's graph to make it agree with Mark's definition in the corrigendum (I think)

\maketitle

\tableofcontents

\section{Introduction.}

${\rm CAT}(0)$ cube complexes provide an ideal setting to study non-positive curvature in a discrete context. On the one hand, their geometry is sufficiently rich to ensure that large classes of groups admit interesting actions on them. Right-angled Artin groups, hyperbolic or right-angled Coxeter groups, hyperbolic $3$-manifold groups, random groups at sufficiently low density all act geometrically on ${\rm CAT}(0)$ cube complexes \cite{Niblo-Reeves, Bergeron-Wise, Ollivier-Wise}, to name just a few examples. Moreover, every finitely generated group with a codimension one subgroup admits an action on a ${\rm CAT}(0)$ cube complex with unbounded orbits \cite{Sageev,Gerasimov,Niblo-Roller}.

On the other hand, the geometry of finite dimensional ${\rm CAT}(0)$ cube complexes is much better understood than general ${\rm CAT}(0)$ geometry, even with no local compactness assumption. For instance, groups acting properly on finite dimensional ${\rm CAT}(0)$ cube complexes are known to have finite asymptotic dimension \cite{Wright}, to satisfy the Von Neumann-Day dichotomy and, if torsion-free, even the Tits alternative \cite{Sageev-Wise, CS}. It is not known whether the same are true for general ${\rm CAT}(0)$ groups.

Three features are particularly relevant in the study of ${\rm CAT}(0)$ cube complexes. First of all, they are endowed with a metric of non-positive curvature. Secondly, the $1$-skeleton becomes a median graph when endowed with its intrinsic path-metric; this means that, for any three vertices, there exists a unique vertex that lies between any two of them. This property is closely related to the existence of hyperplanes and allows for a combinatorial approach that is not available in general ${\rm CAT}(0)$ spaces. Finally, cube complexes are essentially discrete objects, in that their geometry is fully encoded by the $0$-skeleton. In particular, the automorphism group of a cube complex is totally disconnected.

It is natural to wonder how much in the theory of ${\rm CAT}(0)$ cube complexes can be extended to those spaces that share the second feature: \emph{median spaces}. These provide a simultaneous generalisation of ${\rm CAT}(0)$ cube complexes and real trees; for an introduction, see e.g.~\cite{Nica-thesis, CDH, Bow4} and references therein. The class of median spaces is closed under ultralimits and also includes all $L^1$ spaces, so certain pathologies are bound to arise in this context. Bad behaviours seem however to be restricted to spaces of ``infinite rank''.

The notion of \emph{rank} of a median space was introduced in \cite{Bow1}; for ${\rm CAT}(0)$ cube complexes it coincides with the usual concept of dimension. It was recently shown that connected, finite rank median spaces are also endowed with a \emph{canonical} ${\rm CAT}(0)$ metric \cite{Bow4}. Moreover, any finite rank median space can be canonically completed to a connected median space (Corollary~\ref{X''}). Thus it seems that, in finite rank, the only true difference between ${\rm CAT}(0)$ cube complexes and general median spaces lies in the discreteness of the former. Most of the results of the present paper support this analogy.

For our purposes, it will be essential to restrict to median spaces of finite rank. Indeed, we wish to obtain a version of the Tits alternative, while every amenable group admits a proper action on an infinite rank median space \cite{Cherix-Martin-Valette,CDH}. 

Nontrivial examples of group actions on connected median spaces of finite rank are generally obtained by taking limits of actions on ${\rm CAT}(0)$ cube complexes of uniformly bounded dimension. This phenomenon is the higher-dimensional analogue of actions on real trees arising as limits of actions on simplicial trees. For instance, Casals-Ruiz and Kazachkov proved the following generalisation of Rips' Theorem: a finitely generated group $\G$ acts freely, essentially freely and co-specially on a finite rank median space if and only if $\G$ is a subgroup of a graph product of cyclic and surface groups \cite{Casals-Kazachkov}.

Note that one can alternatively take limits of spaces that only ``coarsely'' resemble ${\rm CAT}(0)$ cube complexes. Indeed, ultralimits of hierarchically hyperbolic spaces \cite{Bow3,HHS,HHS2} and coarse median spaces \cite{Bow1} have canonical median metrics of finite rank \cite{Zeidler}. This applies for instance to asymptotic cones of hyperbolic groups, cubulated groups, most irreducible 3-manifold groups and mapping class groups. In the latter case, this was observed already in \cite{Behrstock-Minsky,Behrstock-Drutu-Sapir, Behrstock-Drutu-Sapir2}.

One last example of connected finite rank median spaces is provided by Guirardel cores of pairs of actions of a group $\G$ on real trees \cite{Guirardel}. When $\G$ is a surface group and the two real trees arise from two transverse measured laminations $\mc{L}_1,\mc{L}_2$ on the surface $S$, this space is also known as the Culler-Levitt-Shalen core. The corresponding median space is an $\ell^1$ version of the singular Euclidean metric on $S$ arising from $\mc{L}_1$ and $\mc{L}_2$, lifted to the universal cover. Guirardel's construction generalises to actions of a group $\G$ on any (finite) number of real trees; it is likely that the median spaces arising this way can be studied along the lines of \cite{Hagen-Wilton}.
 
Our main result is the following version of the Tits alternative:

\begin{thmA}
Let $X$ be a complete, finite rank median space. Let $\G$ be a group with an isometric action $\G\acts X$. Suppose that $\G$ has no nonabelian free subgroups.
\begin{enumerate}
\item If the action is free, $\G$ is virtually finite-by-abelian. If moreover $X$ is connected or $\G$ is finitely generated, then $\G$ is virtually abelian.
\item If the action is (metrically) proper, $\G$ is virtually (locally finite)-by-abelian.
\item If all point stabilisers are amenable, $\G$ is amenable.
\end{enumerate}
\end{thmA}

Note that properness of the action cannot be used to conclude that $\G$ is virtually abelian like in part~$(1)$ of Theorem~A. Indeed, every locally finite, countable group admits a proper action on a simplicial tree; see Example~II.7.11 in \cite{BH}. An analogous construction can be used to obtain finitely generated examples: for every finite group $F$, the wreath product $F\wr\Z$ acts properly on the product of a tree and a line.

%\begin{ex}
%Let $F$ be a finite group. We are going to show that the wreath product $F\wr\Z$ acts properly on a ${\rm CAT}(0)$ square complex $X$. The construction is somehow reminiscent of Example~II.7.11 in \cite{BH}.
%
%Every element of $F\wr\Z$ can be written as a pair $((f_k)_{k\in\Z},n)$, where $f_k\in F$ and $n\in\Z$. Set $t:=(\underline{0},1)$; our convention is that $t^i(\underline f,n)t^{-i}=(\tau^i\underline f,n)$, where $\tau^i(f_k)_{k\in\Z}=(f_{k-i})_{k\in\Z}$.
%
%As a set, the $0$-skeleton of $X$ consists of triples $(\underline f,m,n)$, with $\underline f\in\bigoplus_{\Z}F$ and $m,n\in\Z$, up to the following equivalence relation. We identify $(\underline f,m,n)$ and $(\underline f',m,n)$ if $f_k=f_k'$ for every $k\leq m-1$. The wreath product $F\wr\Z$ acts on the $0$-skeleton of $X$ by $(\underline g,p)\cdot(\underline f,m,n)=(\underline g+\tau^p\underline f,m+p,n+p)$.
%
%Join the vertex $(\underline f,m,n)$ to $(\underline f,m+1,n)$ and $(\underline f,m,n\pm 1)$ with three edges. Filling in the resulting squares, we obtain the ${\rm CAT}(0)$ square complex $X$. The action of $F\wr\Z$ extends to $X$. Note that $X$ splits as $T\x\R$, where $T$ is a simplicial tree.
%
%We leave it to the reader to check that $F\wr\Z$ acts properly on $X$ and fixes a point of $\partial X$.
%\end{ex}

To the best of our knowledge, Theorem~A does not yield the Tits alternative for any new families of groups, or at least not any families that one may naturally be led to consider. Indeed, it seems that most groups acting nontrivially on a finite rank median space also act just as nontrivially on a finite dimensional ${\rm CAT}(0)$ cube complex. We are aware of only one exception to this pattern, namely the group $L$ constructed in Section~2 of \cite{Minasyan}. Indeed, $L$ acts with unbounded orbits on a real tree, but every action of $L$ on a finite dimensional ${\rm CAT}(0)$ cube complex must fix a point; the latter property follows from recent results on Thompson's group $V$ \cite{Kato,GenevoisV}.

However, when a group $\G$ does act nontrivially on a ${\rm CAT}(0)$ cube complex, the collection of $\G$-actions on finite rank median spaces tends to be much richer than the collection of cubical $\G$-actions. It is in the study of such actions that the techniques developed in this paper become most useful; see Theorems~B to~F below.

Theorem~A also sheds new light on the relationship between finite and infinite rank median spaces. It shows that actions on finite rank median spaces are a lot more restrictive than actions on general median spaces. Indeed, a discrete group $\G$ has the Haagerup property if and only if it admits a proper action on a median space \cite{Cherix-Martin-Valette,CDH}. Thus, every torsion-free amenable group $\G$ acts properly and freely on a median space, but the latter will never be finite rank if $\G$ is not virtually abelian.

In fact, there even exist groups $\G$ with the Haagerup property, such that every action of $\G$ on a complete, connected, finite rank median space has a global fixed point. A simple example is provided by irreducible lattices in $SL_2\R\x SL_2\R$; see \cite{Fioravanti3}.

Our proof of Theorem~A follows the same broad outline as the corresponding result for ${\rm CAT}(0)$ cube complexes, as it appears in \cite{CS,CFI}. Given an action $\G\acts X$, either $\G$ has a finite orbit within a suitable compactification of $X$, or the space $X$ exhibits a certain tree-like behaviour. In the former case, one obtains a ``big'' abelian quotient of $\G$; in the latter, one can construct many nonabelian free subgroups with a ping-pong argument.

We remark that the first proof of the Tits alternative for ${\rm CAT}(0)$ cube complexes was due to M.~Sageev and D.~T.~Wise \cite{Sageev-Wise} and follows a completely different strategy. It relies on the Algebraic Torus Theorem \cite{Dunwoody-Swenson} and a key fact is that, when a group $\G$ acts nontrivially on a ${\rm CAT}(0)$ cube complex, some hyperplane stabiliser is a codimension one subgroup of $\G$. Unfortunately, this approach is bound to fail when dealing with median spaces: there is an analogous notion of ``hyperplane'' (see Section~\ref{prelims}), but all hyperplane stabilisers could be trivial, even if the group $\G$ is one-ended. This happens for instance when a surface group acts freely on a real tree \cite{Morgan-Shalen}.

Many of the techniques of \cite{CS} have proven extremely useful in the study of ${\rm CAT}(0)$ cube complexes, for instance in \cite{Nevo-Sageev,Fernos,CFI,Kar-Sageev,Kar-Sageev2} to name just a few examples. We extend some of this machinery to median spaces, in particular what goes by the name of ``flipping'', ``skewering'' and ``strong separation''; see Theorems~B and~D below. We will exploit these results in \cite{Fioravanti3} to obtain a superrigidity result analogous to the one in \cite{CFI}.

To state the rest of the results of the present paper, we need to introduce some terminology. In \cite{Fioravanti1} we defined a compactification $\overline X$ of a complete, finite rank median space $X$; we refer to it as the \emph{Roller compactification} of $X$. For ${\rm CAT}(0)$ cube complexes, it consists precisely of the union of $X$ and its Roller boundary $\partial X$ in the usual sense; see \cite{BCGNW, Nevo-Sageev} for a definition. 

Roller compactifications of median spaces strongly resemble Roller compactifications of cube complexes. For instance, the space $\overline X$ has a natural structure of median algebra and $\partial X$ is partitioned as a union of median spaces (``components'' in our terminology), whose rank is strictly lower than that of $X$. These properties will be essential when extending the machinery of \cite{CS} as they enable proofs by induction on the rank. 

For ${\rm CAT}(0)$ cube complexes, our approach is slightly different from that of \cite{CS}, in that we work with Roller boundaries rather than visual boundaries. This is not due to any technical obstructions related to the latter. Rather, we believe that Roller boundaries provide slightly more transparent and elementary proofs even in the context of cube complexes. Indeed, many of our arguments work in any finite rank median algebra.

Let $\G$ be a group. We say that an isometric action $\G\acts X$ is \emph{Roller elementary} if $\G$ has at least one finite orbit in $\overline X$. An action $\G\acts X$ is \emph{Roller minimal} if $X$ is not a single point and $\G$ does not leave invariant any proper, closed, convex subset of $\overline X$; for the notion of convexity in the median algebra $\overline X$, see Section~\ref{prelims}. In ${\rm CAT}(0)$ cube complexes, Roller minimal actions are precisely essential actions (in the terminology of \cite{CS}) with no fixed point in the visual boundary.

Like ${\rm CAT}(0)$ cube complexes, median spaces are endowed with a canonical collection of halfspaces $\mscr{H}$. We say that $\G\acts X$ is \emph{without wall inversions} if there exists no $g\in\G$ such that $g\mf{h}=\mf{h}^*$ for some $\mf{h}\in\mscr{H}$. Here $\mf{h}^*$ denotes the complement $X\setminus\mf{h}$ of the halfspace $\mf{h}$. Actions without wall inversions are a generalisation of actions on $0$-skeleta of simplicial trees without edge inversions. We remark that, perhaps counterintuitively, every action on a \emph{connected} median space is automatically without wall inversions (see Proposition~\ref{all about halfspaces} below); this includes all actions on real trees.

We have the following analogue of the Flipping and Double Skewering Lemmata from \cite{CS}.

\begin{thmB}
Let $X$ be a complete, finite rank median space with a Roller minimal action $\G\acts X$ without wall inversions.
\begin{itemize}
\item For ``almost every'' halfspace $\mf{h}$, there exists $g\in\G$ with $\mf{h}^*\subsetneq g\mf{h}$ and $d(g\mf{h}^*,\mf{h}^*)>0$.
\item For ``almost all'' halfspaces $\mf{h},\mf{k}$ with $\mf{h}\cu\mf{k}$ there exists $g\in\G$ with $g\mf{k}\subsetneq\mf{h}\cu\mf{k}$ and $d(g\mf{k}, \mf{h}^*)>0$.
\end{itemize}
\end{thmB}

Positivity of distances is automatic in cube complexes, but not in general median spaces. Theorem~B cannot be proved for \emph{all} halfspaces; counterexamples already appear in real trees, see Section~\ref{CS-like machinery}. The notion of ``almost every'' should be understood with respect to a certain measure on $\mscr{H}$; see Section~\ref{prelims} below for a definition.

For every complete, finite rank median space $X$, we can define a \emph{barycentric subdivision} $X'$; we study these in Section~\ref{splitting the atom}. The space $X'$ is again complete and median of the same rank. There is a canonical isometric embedding $X\hookrightarrow X'$ and every action on $X$ extends to an action without wall inversions on $X'$. Thus, the assumption that $\G\acts X$ be without wall inversions in Theorem~B is not restrictive. 

Roller minimality may seem a strong requirement, but the structure of Roller compactifications makes Roller minimal actions easy to come by:

\begin{propC}
Let $X$ be a complete, finite rank median space with an isometric action $\G\acts X$. If $\G$ fixes no point of $\overline X$, there exists a $\G$-invariant component $Z\cu\overline X$ and a closed, convex, $\G$-invariant subset $C\cu Z$ such that the action $\G\acts C$ is Roller minimal.
\end{propC}

We remark that actions with a global fixed point in $\overline X$ have a very specific structure, see Theorem~F below.

Theorem~B allows us to construct free groups of isometries, as soon as we have ``tree-like'' configurations of halfspaces inside the median space $X$. More precisely, we need \emph{strongly separated} pairs of halfspaces. In the case of ${\rm CAT}(0)$ cube complexes, these were introduced in \cite{Behrstock-Charney} and used in \cite{CS} to characterise irreducibility. We prove:

\begin{thmD}
Let $X$ be a complete, finite rank median space admitting a Roller minimal action $\G\acts X$ without wall inversions. The median space splits as a nontrivial product if and only if no two halfspaces are strongly separated.
\end{thmD}

Following well-established techniques \cite{CS}, Theorems~B and~D yield:

\begin{thmE}
Let $X$ be a complete, finite rank median space with an isometric action $\G\acts X$. Either $\G$ contains a nonabelian free subgroup or the action is Roller elementary.
\end{thmE}

The last step in the proof of Theorem~A consists of the study of Roller elementary actions; we employ the same strategy as the appendix to \cite{CFI}. To this end, we need to extend the notion of \emph{unidirectional boundary set (UBS)} to median spaces. 

In ${\rm CAT}(0)$ cube complexes, UBS's were introduced in \cite{Hagen}. Up to a certain equivalence relation, they define the simplices in Hagen's simplicial boundary and provide a useful tool to understand Tits boundaries, splittings and divergence \cite{Hagen, Behrstock-Hagen}. They can be thought of as a generalisation of embedded cubical orthants. 

We introduce UBS's in median spaces and use them to prove:

\begin{thmF}
Let $X$ be complete and finite rank. Suppose that $\G\acts X$ fixes a point in the Roller boundary of $X$. A finite-index subgroup $\G_0\leq\G$ fits in an exact sequence
\[1\longrightarrow N\longrightarrow \G_0 \longrightarrow\R^r,\]
where $r=\text{rank}(X)$ and every finitely generated subgroup of $N$ has an orbit in $X$ with at most $2^r$ elements. If $X$ is connected, every finitely generated subgroup of $N$ fixes a point.
\end{thmF}

A more careful study of UBS's will be carried out in \cite{Fioravanti3}, where we show that Roller elementarity is equivalent to the vanishing of a certain cohomology class. In particular we prove that, if $\G$ is a topological group acting with continuous orbits, then $\G_0<\G$ is open and the homomorphism $\G_0\ra\R^r$ in Theorem~F is continuous.

{\bf Structure of the paper.} In Section~\ref{preliminaries} we review the basic theory of median spaces and median algebras, with a special focus on our previous results \cite{Fioravanti1}. We use halfspaces to characterise when a finite rank median space splits as a nontrivial product. We study barycentric subdivisions and a similar construction that allows us to canonically embed finite rank median spaces into connected ones. Section~\ref{elliptic isometries} is concerned with groups of elliptic isometries. In Section~\ref{stab xi} we study UBS's and prove Theorem~F; we also anticipate the proof of Theorem~A, relying on Theorem~E. In Section~\ref{CS-like machinery} we introduce Roller elementarity and Roller minimality; we prove Proposition~C and Theorems~B and~D. Finally, Section~\ref{facing triples} is devoted to constructing free groups of isometries; we prove Theorem~E there.

{\bf Acknowledgements.} The author warmly thanks Brian Bowditch, Pier\-re-Emmanuel Caprace, Indira Chatterji, Thomas Delzant, Cornelia Dru\c tu, Talia Fern\'os, Mark Hagen for many helpful conversations and Anthony Genevois for his comments on an earlier version. The author expresses special gratitude to Cornelia Dru\c tu for her excellent supervision and to Talia Fern\'os for her encouragement to pursue this project. The author also wishes to thank the anonymous referee for their valuable comments.

This work was undertaken at the Mathematical Sciences Research Institute in Berkeley during the Fall 2016 program in Geometric Group Theory, where the author was supported by the National Science Foundation under Grant no.~DMS-1440140 and by the GEAR Network. Part of this work was also carried out at the Isaac Newton Institute for Mathematical Sciences, Cambridge, during the programme ``Non-positive curvature, group actions and cohomology'' and was supported by EPSRC grant no.~EP/K032208/1. The author was also supported by the Clarendon Fund and the Merton Moussouris Scholarship.

\section{Preliminaries.}\label{preliminaries}

\subsection{Median spaces and median algebras.}\label{prelims}

Let $X$ be a metric space. A finite sequence of points $(x_k)_{1\leq k\leq n}$ is a \emph{geodesic} if $d(x_1,x_n)=d(x_1,x_2)+...+d(x_{n-1},x_n)$. The interval $I(x,y)$ between $x,y\in X$ is the set of points lying on a geodesic from $x$ to $y$. We say that $X$ is a \emph{median space} if, for all $x,y,z\in X$, the intersection $I(x,y)\cap I(y,z)\cap I(z,x)$ consists of a single point, which we denote by $m(x,y,z)$. In this case, the map $m\colon X^3\ra X$ endows $X$ with a structure of \emph{median algebra}, see e.g.~\cite{CDH, Bow1, Roller} for a definition and various results.

In a median algebra $(M,m)$, the \emph{interval} $I(x,y)$ between $x,y\in M$ is the set of points $z\in M$ with $m(x,y,z)=z$; this is equivalent to the definition above if $M$ arises from a median space. A subset $C\cu M$ is \emph{convex} if $I(x,y)\cu C$ whenever $x,y\in C$. Every collection of pairwise intersecting convex subsets of a median algebra has the finite intersection property; this is Helly's Theorem, see e.g.~Theorem~2.2 in \cite{Roller}.

A \emph{halfspace} is a convex subset $\mf{h}\cu M$ whose complement $\mf{h}^*:=M\setminus\mf{h}$ is also convex. We will refer to the unordered pair $\mf{w}:=\{\mf{h},\mf{h}^*\}$ as a \emph{wall}; we say that $\mf{h}$ and $\mf{h}^*$ are the \emph{sides} of $\mf{w}$. The wall $\mf{w}$ separates subsets $A\cu M$ and $B\cu M$ if $A\cu\mf{h}$ and $B\cu\mf{h}^*$ or vice versa. The wall $\mf{w}$ is \emph{contained} in a halfspace $\mf{k}$ if $\mf{h}\cu\mf{k}$ or $\mf{h}^*\cu\mf{k}$. We say, with a slight abuse of terminology, that $\mf{w}$ is contained in $\mf{k}_1\cap...\cap\mf{k}_k$ if $\mf{w}$ is contained in $\mf{k}_i$ for each $i$.

The sets of halfspaces and walls of $M$ are denoted $\mscr{H}(M)$ and $\mscr{W}(M)$ respectively, or simply $\mscr{H}$ and $\mscr{W}$. Given subsets $A,B\cu M$, we write $\mscr{H}(A|B)$ for the set of halfspaces with $B\cu\mf{h}$ and $A\cu\mf{h}^*$ and we set $\s_A:=\mscr{H}(\emptyset|A)$; we will not distinguish between $x\in M$ and the singleton $\{x\}$. If $A,B$ are convex and disjoint, we have $\mscr{H}(A|B)\neq\emptyset$, see e.g.~Theorem~2.7 in \cite{Roller}. We will refer to the sets $\mscr{H}(x|y)$, $x,y\in M$, as \emph{halfspace-intervals}.

A \emph{pocset} $(\mscr{P},\preceq,*)$ consists of a poset $(\mscr{P},\preceq)$ equipped with an order-reversing involution $*$, such that every element $a\in\mscr{P}$ is incomparable with $a^*\in\mscr{P}$. If $E\cu\mscr{P}$, we write $E^*:=\{a^*\mid a\in E\}$. Ordering $\mscr{H}$ by inclusion, we obtain a pocset where the involution is given by taking complements. 

We say that $a,b\in\mscr{P}$ are \emph{transverse} if any two elements in the set $\{a,a^*,b,b^*\}$ are incomparable. Halfspaces $\mf{h},\mf{k}\in\mscr{H}$ are transverse if and only if the intersections $\mf{h}\cap\mf{k}$, $\mf{h}\cap\mf{k}^*$, $\mf{h}^*\cap\mf{k}$, $\mf{h}^*\cap\mf{k}^*$ are all nonempty subsets of $M$. We say that two walls are \emph{transverse} if they arise from transverse halfspaces. 

A subset $\s\cu\mscr{P}$ is a \emph{partial filter} if there do not exist $a,b\in\s$ with $a\preceq b^*$. A partial filter $\s$ is an \emph{ultrafilter} if $\mscr{P}=\s\cup\s^*$.

\begin{lem}\label{Zorn}
Let $(\mscr{P},\preceq,*)$ be a pocset. Let $\underline{\s}\cu\mscr{P}$ be a partial filter and let $\overline{\s}\cu\mscr{P}$ contain $\mscr{P}\setminus(\underline{\s}\cup\underline{\s}^*)$. If $\underline{\s}\cu\overline{\s}$, then there exists an ultrafilter $\s$ with $\underline{\s}\cu\s\cu\overline{\s}$. 
\end{lem}
\begin{proof}
By Zorn's Lemma, there exists a maximal partial filter $\s$ satisfying $\underline{\s}\cu\s\cu\overline{\s}$. We now show that $\mscr{P}=\s\cup\s^*$, hence $\s$ is an ultrafilter. 

Suppose for the sake of contradiction that there exists $a\not\in\s\cup\s^*$. By maximality of $\s$, neither $\s\cup\{a\}$ nor $\s\cup\{a^*\}$ can be a partial filter. Thus, there exist $b_1,b_2\in\s$ with $a\preceq b_1^*$ and $a^*\preceq b_2^*$. In particular $b_2\preceq a\preceq b_1^*$, contradicting the fact that $\s$ is a partial filter.
\end{proof}

According to the terminology above, a subset $\s\cu\mscr{H}$ is a partial filter if and only if it consists of pairwise intersecting halfspaces. It is an ultrafilter if and only if it moreover contains a side of every wall of $M$. For every $x\in M$, the subset $\s_x\cu\mscr{H}$ is an ultrafilter. Note that, by the proof of Lemma~\ref{Zorn}, ultrafilters are precisely maximal sets of pairwise intersecting halfspaces.

 A subset $\Om\cu\mscr{H}$ is said to be \emph{inseparable} if it contains all $\mf{j}\in\mscr{H}$ such that $\mf{h}\cu\mf{j}\cu\mf{k}$, for $\mf{h},\mf{k}\in\Om$. The \emph{inseparable closure} of a subset $\Om\cu\mscr{H}$ is the smallest inseparable set containing $\Om$; it coincides the union of the sets $\mscr{H}(\mf{k}^*|\mf{h})$ as $\mf{h},\mf{k}$ vary in $\Om$.

The set $\{-1,1\}$ has a unique structure of median algebra. Considering its median map separately in all coordinates, we endow $\{-1,1\}^k$ with a median-algebra structure for each $k\in\N$; we will refer to it as a \emph{$k$-hypercube}. The \emph{rank} of $M$ is the maximal $k\in\N$ such that we can embed a $k$-hypercube into $M$. By Proposition~6.2 in \cite{Bow1} this is the same as the maximal cardinality of a set of pairwise-transverse halfspaces. Note that $M$ has rank zero if and only if it consists of a single point. The following is immediate from Ramsey's Theorem \cite{Ramsey}:

\begin{lem}\label{Ramsey}
If $M$ has finite rank and $\s_1,\s_2\cu\mscr{H}$ are two ultrafilters, every infinite subset of $\s_1\setminus\s_2$ contains an infinite subset that is totally ordered by inclusion.
\end{lem}

If $C\cu M$ and $x\in M$, we say that $y\in C$ is a \emph{gate} for $(x,C)$ if $y\in I(x,z)$ for all $z\in C$. The set $C$ is \emph{gate-convex} if a gate exists for every point of $M$; in this case, gates are unique and define a \emph{gate-projection} $\pi_C\colon M\ra C$. If $C$ is gate-convex, we have $\mscr{H}(x|\pi_C(x))=\mscr{H}(x|C)$ for every $x\in M$. Every interval $I(x,y)$ is gate-convex with projection $z\mapsto m(x,y,z)$. Gate-convex sets are always convex, but the converse does not hold in general. We record a few more properties of gate-convex subsets in the following result; see \cite{Fioravanti1} for proofs.

\begin{prop}\label{all about gates}
Let $C,C'\cu M$ be gate-convex.
\begin{enumerate}
\item The sets $\{\mf{h}\in\mscr{H}(M)\mid\mf{h}\cap C\neq\emptyset,~\mf{h}^*\cap C\neq\emptyset\}$, ${\{\pi_C^{-1}(\mf{h})\mid\mf{h}\in\mscr{H}(C)\}}$ and $\mscr{H}(C)$ are all naturally in bijection.
\item There exists a \emph{pair of gates}, i.e.~a pair $(x,x')$ of points $x\in C$ and $x'\in C'$ such that $\pi_C(x')=x$ and $\pi_{C'}(x)=x'$. In particular, we have ${\mscr{H}(x|x')=\mscr{H}(C|C')}$.
\item If $C\cap C'\neq\emptyset$, we have $\pi_C(C')=C\cap C'$ and $\pi_C\o\pi_{C'}=\pi_{C'}\o\pi_C$. In particular, if $C'\cu C$, we have $\pi_{C'}=\pi_{C'}\o\pi_C$.
\end{enumerate}
\end{prop}

A \emph{topological median algebra} is a median algebra endowed with a Hausdorff topology so that $m$ is continuous. Every median space $X$ is a topological median algebra, since $m\colon X^3\ra X$ is $1$-Lipschitz if we consider the $\ell^1$ metric on $X^3$. Every gate-convex subset of a median space is closed and convex; the converse holds if $X$ is complete. Gate-projections are $1$-Lipschitz. If ${C,C'\cu X}$ are gate-convex and $X$ is complete, points $x\in C$ and $x'\in C'$ form a pair of gates if and only if $d(x,x')=d(C,C')$, see Lemma~2.9 in \cite{Fioravanti1}; this holds in particular when $C'$ is a singleton. 

Let $X$ be a complete, finite rank median space. The following is Proposition~B in \cite{Fioravanti1}.

\begin{prop}\label{all about halfspaces}
Every halfspace is either open or closed (possibly both). If $\mf{h}_1\supsetneq ... \supsetneq\mf{h}_k$ is a chain of halfspaces with $\overline{\mf{h}_1^*}\cap\overline{\mf{h}_k}\neq\emptyset$, we have $k\leq 2\cdot\text{rank}(X)$.
\end{prop}

\begin{lem}\label{intersections of halfspaces}
Let $\mc{C}\cu\mscr{H}$ be totally ordered by inclusion and suppose that the halfspaces in $\mc{C}$ are at uniformly bounded distance from a point $x\in X$. The intersection of all halfspaces in $\mc{C}$ is nonempty. 
\end{lem}
\begin{proof}
It suffices to consider the case when $\mc{C}$ does not have a minimum. By Lemma~2.27 in \cite{Fioravanti1}, we can find a cofinal subset $\{\mf{h}_n\}_{n\geq 0}\cu\mc{C}$ with $\mf{h}_{n+1}\subsetneq\mf{h}_n$; in other words, every $\mf{k}\in\mc{C}$ contains a halfspace of the form $\mf{h}_n$. 

Let $x_n$ be the gate-projection of $x$ to $\overline{\mf{h}_n}$. By Proposition~\ref{all about gates}, the sequence $(x_n)_{n\geq 0}$ is Cauchy; hence it converges to a point $\overline x\in X$, which lies in $\overline{\mf{h}_n}$ for all $n\geq 0$. If $\overline x$ did not lie in every $\mf{h}_n$, there would exist $N\geq 0$ with $\overline x\in\overline{\mf{h}_n}\setminus\mf{h}_n$ for all $n\geq N$. In particular, $\overline{\mf{h}_n^*}\cap\overline{\mf{h}_m}\neq\emptyset$ for all $m,n\geq N$ and this would violate Proposition~\ref{all about halfspaces}.
\end{proof}

Endowing $\R^n$ with the $\ell^1$ metric, we obtain a median space. Like $\R^n$, a rich class of median spaces also has an analogue of the $\ell^{\infty}$ metric (see \cite{Bow5}) and of the $\ell^2$ metric (see \cite{Bow4}). We record the following result for later use.

\begin{thm}[\cite{Bow4}]\label{CAT(0) metric}
If $X$ is connected, it admits a bi-Lipschitz-equivalent ${\rm CAT}(0)$ metric that is canonical in the sense that:
\begin{itemize}
\item every isometry of the median metric of $X$ is also an isometry for the ${\rm CAT}(0)$ metric;
\item the ${\rm CAT}(0)$ geodesic between $x$ and $y$ is contained in $I(x,y)$; in particular, subsets that are convex for the median metric are also convex for the ${\rm CAT}(0)$ metric.
\end{itemize}
\end{thm}

A \emph{pointed measured pocset (PMP)} is a 4-tuple $\left(\mscr{P},\mscr{D},\eta,\s\right)$, where $\mscr{P}$ is a pocset, $\s\cu\mscr{P}$ is an ultrafilter, $\mscr{D}$ is a $\s$-algebra of subsets of $\mscr{P}$ and $\eta$ is a measure defined on $\mscr{D}$. Let $\overline{\mc{M}}\left(\mscr{P},\mscr{D},\eta\right)$ be the set of all ultrafilters ${\s'\cu\mscr{P}}$ with $\s'\triangle\s\in\mscr{D}$, where we identify sets with $\eta$-null symmetric difference. We endow this space with the extended metric $d(\s_1,\s_2):=\eta(\s_1\triangle\s_2)$. The set of points at finite distance from $\s$ is a median space, which we denote $\mc{M}\left(\mscr{P},\mscr{D},\eta,\s\right)$; see Section~2.2 in \cite{Fioravanti1}.

Let $X$ be a complete, finite rank median space. In \cite{Fioravanti1} we constructed a semifinite measure $\wh{\nu}$ defined on a $\s$-algebra $\wh{\mscr{B}}\cu 2^{\mscr{H}}$ such that $\wh{\nu}(\mscr{H}(x|y))=d(x,y)$, for all $x,y\in X$. There, we referred to elements of $\wh{\mscr{B}}$ as \emph{morally measurable sets}, but, for the sake of simplicity, we will just call them \emph{measurable sets} here. Note that this measure space is different from the ones considered in \cite{CDH}. 

Every inseparable subset of $\mscr{H}$ lies in $\wh{\mscr{B}}$; in particular, all ultrafilters on $\mscr{H}$ are measurable. If $C,C'\cu X$ are convex (or empty), the set $\mscr{H}(C|C')$ is measurable and $\wh{\nu}(\mscr{H}(C|C'))=d(C,C')$. A halfspace is an atom for $\wh{\nu}$ if and only if it is clopen. The space $X$ is connected if and only if $\wh{\nu}$ has no atoms, in which case $X$ is geodesic. We say that a halfspace $\mf{h}$ is \emph{thick} if both $\mf{h}$ and $\mf{h}^*$ have nonempty interior; $\wh{\nu}$-almost every halfspace is thick. We denote by $\mscr{H}^{\x}$ the set of non-thick halfspaces. See \cite{Fioravanti1} for proofs.

Picking a basepoint $x\in X$, we can identify $X\simeq\mc{M}(\mscr{H},\wh{\mscr{B}},\wh{\nu},\s_x)$ isometrically by mapping each $y\in X$ to the ultrafilter $\s_y\cu\mscr{H}$, see Corollary~3.12 in \cite{Fioravanti1}. In particular, $X$ sits inside the space $\overline{\mc{M}}(\mscr{H},\wh{\mscr{B}},\wh{\nu})$, which we denote by $\overline X$. If $I\cu X$ is an interval, we have a projection $\pi_I\colon\overline X\ra I$ that associates to each ultrafilter $\s\cu\mscr{H}$ the only point of $I$ that is represented by the ultrafilter $\s\cap\mscr{H}(I)$. We give $\overline X$ the coarsest topology for which all the projections $\pi_I$ are continuous. Defining
\[m(\s_1,\s_2,\s_3):=(\s_1\cap\s_2)\cup(\s_2\cap\s_3)\cup(\s_3\cap\s_1),\]
we endow $\overline X$ with a structure of topological median algebra. We have:

\begin{prop}[\cite{Fioravanti1}]
The topological median algebra $\overline X$ is compact. The inclusion $X\hookrightarrow\overline X$ is a continuous morphism of median algebras with dense, convex image.
\end{prop}

We call $\overline X$ the \emph{Roller compactification} of $X$ and $\partial X:=\overline X\setminus X$ the \emph{Roller boundary}. We remark that, in general, $X\hookrightarrow\overline X$ is not an embedding and $\partial X$ is not closed in $\overline X$. 

If $C\cu X$ is convex, the closure of $C$ in $X$ coincides with the intersection of $X$ and the closure of $C$ in $\overline X$. If $C\cu X$ is closed and convex, the closure of $C$ inside $\overline X$ is canonically identified with the Roller compactification $\overline C$. The median map of $\overline X$ and the projections $\pi_I\colon\overline X\ra I$ are $1$-Lipschitz with respect to the extended metric on $\overline X$. 

Looking at pairs of points of $\overline X$ at finite distance, we obtain a partition of $\overline X$ into \emph{components}; each component is a median space with the restriction of the extended metric of $\overline X$. The subset $X\cu\overline X$ always forms an entire component of $\overline X$.

\begin{prop}[\cite{Fioravanti1}]\label{components}
Each component $Z\cu\partial X$ is a complete median space with $\text{rank}(Z)\leq\text{rank}(X)-1$. Moreover, $Z$ is convex in $\overline X$ and the inclusion $Z\hookrightarrow\overline X$ is continuous. The closure of $Z$ in $\overline X$ is canonically identified with the Roller compactification $\overline Z$ and there is a gate-projection $\pi_Z\colon\overline X\ra\overline Z$ that maps $X$ into $Z$.
\end{prop}

Every halfspace $\mf{h}\in\mscr{H}$ induces a halfspace $\wt{\mf{h}}$ of $\overline X$ with $\wt{\mf{h}}\cap X=\mf{h}$; thus, we can identify $\mscr{H}$ with a subset of $\mscr{H}(\overline X)$. 

\begin{prop}[\cite{Fioravanti1}]\label{halfspaces of components}
Every thick halfspace of a component $Z\cu\partial X$ is of the form $\wt{\mf{h}}\cap Z$, for some $\mf{h}\in\mscr{H}$.
\end{prop}

Any two points of $\overline X$ are separated by a halfspace of the form $\wt{\mf{h}}$. To every $\xi\in\overline X$, we can associate a \emph{canonical ultrafilter} $\s_{\xi}\cu\mscr{H}$ representing $\xi$; it satisfies $\mf{h}\in\s_{\xi}\Leftrightarrow\xi\in\wt{\mf{h}}$.

\subsection{Products.}

Given median spaces $X_1,X_2$, we can consider the product $X_1\x X_2$, which is itself a median space with the $\ell^1$ metric, i.e.
\[d_{X_1\x X_2}\left((x_1,x_2),(x'_1,x'_2)\right):=d_{X_1}(x_1,x'_1)+d_{X_2}(x_2,x'_2).\]
The space $X_1\x X_2$ is complete if and only if $X_1$ and $X_2$ are. We say that subsets $A,B\cu\mscr{H}$ are \emph{transverse} if $\mf{h}$ and $\mf{k}$ are transverse whenever $\mf{h}\in A$ and $\mf{k}\in B$. We have the following analogue of Lemma~2.5 in \cite{CS}.

\begin{prop}\label{products}
The following are equivalent for complete, finite rank median spaces:
\begin{enumerate}
\item $X$ splits as a product $X_1\x X_2$, where each $X_i$ has at least two points;
\item there is a measurable, $*$-invariant partition $\mscr{H}=\mscr{H}_1\sqcup\mscr{H}_2$, where the $\mscr{H}_i$ are nonempty and transverse;
\item there is measurable, $*$-invariant partition $\mscr{H}=\mscr{H}_1\sqcup\mscr{H}_2\sqcup\mscr{K}$, where the $\mscr{H}_i$ are nonempty and transverse, while $\mscr{K}$ is null.
\end{enumerate}
\end{prop}
\begin{proof}
If $X=X_1\x X_2$, part~$(1)$ of Proposition~\ref{all about gates} provides subsets $\mscr{H}(X_1)$ and $\mscr{H}(X_2)$ of $\mscr{H}(X)$. They are nonempty, disjoint, $*$-invariant, transverse and measurable. Every $\mf{h}\in\mscr{H}(X)$ must split a fibre $X_1\x\{*\}$ or $\{*\}\x X_2$ nontrivially and thus lies either in $\mscr{H}(X_1)$ or in $\mscr{H}(X_2)$.

We conclude by proving that $(3)$ implies $(1)$, since $(2)$ trivially implies $(3)$. Let $\wh{\mscr{B}}_i$ be the $\s$-algebra of subsets of $\mscr{H}_i$ that lie in $\wh{\mscr{B}}$; fixing $x\in X$, we simply write $\mc{M}_i$ for $\mc{M}(\mscr{H}_i,\wh{\mscr{B}}_i,\wh{\nu},\s_x\cap\mscr{H}_i)$. Define a map $\iota\colon X\ra\mc{M}_1\x\mc{M}_2$ by intersecting ultrafilters on $\mscr{H}$ with each $\mscr{H}_i$. Since $\mscr{K}$ is null, this is an isometric embedding. Given ultrafilters $\s_i\cu\mscr{H}_i$, the set $\s_1\sqcup\s_2$ consists of pairwise intersecting halfspaces and is therefore contained in an ultrafilter $\s\cu\mscr{H}$; see Lemma~\ref{Zorn}. Recall that $\s\in\wh{\mscr{B}}$ as $X$ has finite rank. We conclude that $\iota$ is surjective and $X\simeq\mc{M}_1\x\mc{M}_2$.
%until here would have worked for locally convex, avoiding Corollary~\ref{filters are morally measurable}

We are left to show that each $\mc{M}_i$ contains at least two points. We construct points $x_i,x_i'\in X$ such that $\wh{\nu}(\mscr{H}(x_i|x_i')\cap\mscr{H}_i)>0$. Pick any $\mf{h}_i\in\mscr{H}_i$; by Proposition~\ref{all about halfspaces}, replacing $\mf{h}_i$ with its complement if necessary, we can assume that there exists $x_i\not\in\overline{\mf{h}_i}$. Let $x_i'$ be the gate-projection of $x_i$ to $\overline{\mf{h}_i}$. None of the halfspaces in $\mscr{H}(x_i|x_i')$ is transverse to $\mf{h}_i$, thus $\mscr{H}(x_i|x_i')\setminus\mscr{K}$ is contained in $\mscr{H}_i$ and has positive measure.
\end{proof}

In particular, $\text{rank}(X_1\x X_2)=\text{rank}(X_1)+\text{rank}(X_2)$. We say that $X$ is \emph{irreducible} if it cannot be split nontrivially as a product $X_1\x X_2$.  We remark that in parts~$(2)$ and~$(3)$ of Proposition~\ref{products}, the sets $\mscr{H}_i$ are \emph{not} required to have positive measure, but simply to be nonempty. The following is immediate from Proposition~\ref{products}:

\begin{lem}\label{Roller for products}
If $X=X_1\x X_2$, we have $\overline X=\overline X_1\x\overline X_2$.
\end{lem}

We can also use Proposition~\ref{products} to characterise isometries of products. The following is an analogue of Proposition~2.6 in \cite{CS} (also compare with \cite{Foertsch-Lytchak}, when $X$ is connected and locally compact).

\begin{prop}\label{isometries of products}
If $X$ is complete and finite rank, there exists a canonical decomposition $X= X_1\x ... \x X_k$, where each $X_i$ is irreducible. Every isometry of $X$ permutes the factors $X_i$; in particular, the product of the isometry groups of the factors has finite index in $\text{Isom}~X$.
\end{prop}
\begin{proof}
The existence of such a splitting follows from the observation that, in any nontrivial product, factors have strictly lower rank. By Proposition~\ref{products}, this corresponds to a transverse decomposition $\mscr{H}=\mscr{H}_1\sqcup ...\sqcup \mscr{H}_k$, where we can identify $\mscr{H}_i=\mscr{H}(X_i)$. Given $g\in\text{Isom}~X$, the decompositions
\[\mscr{H}_i=\bigsqcup_{j=1}^k\mscr{H}_i\cap g\mscr{H}_j\]
are transverse and each piece is measurable and $*$-invariant. Since $X_i$ is irreducible, we must have $\mscr{H}_i\cap g\mscr{H}_j=\emptyset$ for all but one $j$, again by Proposition~\ref{products}, and the result follows.
\end{proof}

\subsection{Splitting the atom.}\label{splitting the atom}

In this section we describe two constructions that allow us to embed median spaces into ``more connected'' ones. We will only consider complete, finite rank spaces.

Given a median space $X$, let $\mscr{A}\cu\mscr{H}$ be the set of atoms of $\wh{\nu}$. The idea is to split every atom into two ``hemiatoms'' of half the size. We thus obtain a new measured pocset $(\mscr{H}',\mscr{B}',\nu')$, whose associated median space generalises barycentric subdivisions of cube complexes. We now describe this construction more in detail.

As a set, $\mscr{H}'$ consists of $\mscr{H}\setminus\mscr{A}$, to which we add two copies $\mf{a}_+,\mf{a}_-$ of every $\mf{a}\in\mscr{A}$. We have a projection $p\colon\mscr{H}'\ra\mscr{H}$ with fibres of cardinality one or two. We give $\mscr{H}'$ a structure of poset by declaring that $\mf{j}\subsetneq\mf{j}'$ if $p(\mf{j})\subsetneq p(\mf{j}')$, or $\mf{j}=\mf{a}_-$ and $\mf{j}'=\mf{a}_+$ for some $\mf{a}\in\mscr{A}$. We promote this to a structure of pocset by setting $\mf{j}^*=\mf{j}'$ if $p(\mf{j})^*=p(\mf{j}')\not\in\mscr{A}$ and, in addition, $(\mf{a}_-)^*=(\mf{a}^*)_+$, $(\mf{a}_+)^*=(\mf{a}^*)_-$ if $\mf{a}\in\mscr{A}$. Observe that each intersection between $\mscr{A}$ and a halfspace-interval is at most countable; thus $\mscr{A}$ and all its subsets are measurable. In particular $\mscr{B}':=\{E\cu\mscr{H}'\mid p(E)\setminus\mscr{A}\in\wh{\mscr{B}}\}$ is a $\s$-algebra of subsets of $\mscr{H}'$, on which we can define the measure
\[\nu'(E):=\wh{\nu}\left(p(E)\setminus\mscr{A}\right)+\frac{1}{2}\cdot\sum_{\substack{\mf{a}\in\mscr{A} \\ \mf{a}_+\in E}}\wh{\nu}\left(\{\mf{a}\}\right)+\frac{1}{2}\cdot\sum_{\substack{\mf{a}\in\mscr{A} \\ \mf{a}_-\in E}}\wh{\nu}\left(\{\mf{a}\}\right).\]
If $F\cu\mscr{H}$ is measurable, we have $\nu'(p^{-1}(F))=\wh{\nu}(F)$. Given $z\in X$, we set $X':=\mc{M}(\mscr{H}',\mscr{B}',\nu',p^{-1}(\s_z))$; note that this does not depend on the choice of $z$. Taking preimages under $p$ of ultrafilters on $\mscr{H}$, we obtain an isometric embedding $X\hookrightarrow X'$.

\begin{lem}\label{cubes in X'}
For each $x\in X'\setminus X$ there exist canonical subsets $C(x)\cu X$, $\wh{C}(x)\cu X'$ and isomorphisms of median algebras
\[\iota_x\colon\{-1,1\}^k\ra C(x),\]
\[\hat{\iota}_x\colon\{-1,0,1\}^k\ra\wh{C}(x).\]
Here $1\leq k\leq r:=\text{rank}(X)$ and the map $\hat{\iota}_x$ extends $\iota_x$ taking $(0,...,0)$ to $x$. Moreover, $C(x)$ is gate-convex in $X$ and $\wh{C}(x)$ is gate-convex in $X'$.
\end{lem}
\begin{proof}
Let $\s\cu\mscr{H}'$ be an ultrafilter representing $x$. Since $x\not\in X$ the set 
\[ W(x):=\left\{\{\mf{a},\mf{a}^*\}\in\mscr{W}(X)\mid \mf{a}\in\mscr{A} ~~\text{and}~~ \mf{a}_+\in\s ~~\text{and}~~ (\mf{a}^*)_+\in\s\right\}\] 
is nonempty. Any two walls in $W(x)$ are transverse, so $k:=\# W(x)\leq r$. Choose halfspaces $\mf{a}_1,...,\mf{a}_k\in\mscr{H}$ representing all walls in $W(x)$. The set $p(\s)\setminus\{\mf{a}_1^*,...,\mf{a}_k^*\}$ is an ultrafilter on $\mscr{H}$ and it represents a point $q\in X$. This will be the point $\iota_x(1,...,1)$ in $C(x)$. To construct $q'\in C(x)$, simply replace $\mf{a}_i\in\s_q\cu\mscr{H}$ with $\mf{a}_i^*$ whenever the $i$-th coordinate of $q'$ is $-1$; the result is an ultrafilter on $\mscr{H}$ representing $q'\in X$.

To construct a point $u\in\wh{C}(x)\cu X'$, consider the point $u'\in C(x)$ obtained by replacing all the zero coordinates of $u$ with $1$'s. Whenever the $i$-th coordinate of $u$ is $0$, we replace $(\mf{a}_i)_-\in p^{-1}(\s_{u'})$ with $(\mf{a}_i^*)_+$, obtaining an ultrafilter on $\mscr{H}'$ that represents the point $u$. 

We are left to check that $C(x)$ and $\wh{C}(x)$ are gate-convex. Let $H(x)\cu\mscr{H}$ be the set of halfspaces corresponding to the walls of $W(x)$. We define a map $\pi\colon X'\ra\wh{C}(x)$ as follows: given an ultrafilter $\s'\cu\mscr{H}'$, the intersection $\s'\cap p^{-1}(H(x))$ determines a unique point of $\wh{C}(x)$ and we call this $\pi(\s')$. Note that the restriction of $\pi$ to $X$ takes values in $C(x)$. It is straightforward to check that $\pi$ and $\pi|_X$ are gate-projections.
\end{proof}

When $x\in X$, we set $C(x)=\wh{C}(x)=\{x\}$. 

\begin{lem}\label{halfspaces of X'}
Every halfspace of $X'$ arises from an element of $\mscr{H}'$.
\end{lem}
\begin{proof}
Observe that $\text{Hull}_{X'}(X)=X'$ since, for every $x\in X'$, the hull of $C(x)$ in $X'$ is $\wh{C}(x)$. Thus, every halfspace of $X'$ intersects $X$ in a halfspace of $X$. Given $\mf{h}\in\mscr{H}$, we consider $\mscr{F}(\mf{h}):=\left\{\mf{k}\in\mscr{H}(X')\mid \mf{k}\cap X=\mf{h}\right\}$; note that $\mscr{F}(\mf{h})\neq\emptyset$ by Lemma~6.5 in \cite{Bow1}. If $\mf{h}\in\mscr{A}$, we can construct halfspaces of $X'$ corresponding to $\mf{h}_+,\mf{h}_-\in\mscr{H}'$. For instance, $\mf{h}_+$ corresponds to the set of ultrafilters on $\mscr{H}'$ that contain $\mf{h}_+$; this is well-defined as $\mf{h}_+$ has positive measure. Thus, we only need to show that $\#\mscr{F}(\mf{h})=1$ if $\mf{h}\in\mscr{H}\setminus\mscr{A}$ and $\mscr{F}(\mf{h})=\{\mf{h}_-,\mf{h}_+\}$ if $\mf{h}\in\mscr{A}$.

If, for some $\mf{k}\in\mscr{H}(X')$ and $x\in X'$, both $\mf{k}\cap\wh{C}(x)$ and $\mf{k}^*\cap\wh{C}(x)$ are nonempty, there exists $\mf{h}\in\mscr{A}$ such that $\mf{k}\in\{\mf{h}_-,\mf{h}_+\}$, by part~$(1)$ of Proposition~\ref{all about gates}. Thus, we can suppose that, for every $x\in X'$, we either have $\wh{C}(x)\cu\mf{k}$ or $\wh{C}(x)\cu\mf{k}^*$. Suppose, for the sake of contradiction, that the same is true of $\mf{k}'\in\mscr{H}(X')$, with $\mf{k}'\neq\mf{k}$ and $\mf{k}'\cap X=\mf{k}\cap X$. Let $z\in\mf{k}\triangle\mf{k}'$ be a point; observe that $z\not\in X$ and, by our assumptions, the hypercube $\wh{C}(z)$ is entirely contained in $\mf{k}\triangle\mf{k}'$. This implies that $C(z)\cu\mf{k}\triangle\mf{k}'$, violating the fact that $\mf{k}'\cap X=\mf{k}\cap X$.
\end{proof}

Given a subgroup $\G\leq\text{Isom}~X$, we say that the action $\G\acts X$ is \emph{without wall inversions} if $g\mf{h}\neq\mf{h}^*$ for every $g\in\G$ and $\mf{h}\in\mscr{H}$. We denote by $a(X)$ the supremum of the $\wh{\nu}$-masses of the elements of $\mscr{A}$.

\begin{prop}[Properties of $X'$]\label{properties of X'}
\begin{enumerate}
\item The median space $X'$ is complete and $\text{rank}(X')=\text{rank}(X)$.
\item There is an isometric embedding $X\hookrightarrow X'$ and $\text{Hull}_{X'}(X)=X'$.
\item Every isometry of $X$ extends canonically to an isometry of $X'$ yielding $\text{Isom}~X\hookrightarrow\text{Isom}~X'$. Moreover, the induced action $\text{Isom}~X\acts X'$ is without wall inversions. 
\item We have $a(X')\leq\frac{1}{2}\cdot a(X)$.
\item The inclusion $X\hookrightarrow X'$ canonically extends to a monomorphism of median algebras $\overline X\hookrightarrow\overline{X'}$. For every $\xi\in\overline{X'}\setminus\overline X$ there exists a canonical cube $\{-1,0,1\}^k\hookrightarrow\overline{X'}$ centred at $\xi$, with $\{-1,1\}^k\hookrightarrow\overline{X}$ and $1\leq k\leq\text{rank}(X)$.
\end{enumerate}
\end{prop}
\begin{proof}
We have already shown part~$(2)$ and part~$(4)$ is immediate. Part~$(5)$ can be proved like Lemma~\ref{cubes in X'} since we now have Lemma~\ref{halfspaces of X'}. If ${g\in\text{Isom}~X}$ and $\mf{h}\in\mscr{H}$ satisfy $g\mf{h}=\mf{h}^*$, Proposition~\ref{all about halfspaces} implies that the halfspace $\mf{h}$ is clopen; as clopen halfspaces are atoms, part~$(3)$ immediately follows. By Lemma~\ref{halfspaces of X'}, the only nontrivial statement in part~$(1)$ is the completeness of $X'$, which we now address.

Let $\s_n\cu\mscr{H}'$ be ultrafilters corresponding to a Cauchy sequence in $X'$. The sets $\underline{\s}:=\liminf\s_n$ and $\overline{\s}:=\limsup\s_n$ lie in $\mscr{B}'$. Any two halfspaces in $\underline{\s}$ intersect and $\overline{\s}$ contains the complement of $\underline{\s}\cup\underline{\s}^*$. Lemma~\ref{Zorn} implies that $\underline{\s}$ is contained in an ultrafilter $\s\cu\mscr{H}'$ with $\s\cu\overline{\s}$. If we show that $\nu'(\overline{\s}\setminus\underline{\s})=0$, it follows that $\s\in\mscr{B}'$ and the points of $X'$ represented by $\s_n$ converge to the point of $X'$ represented by $\s$. Note that it suffices to show that a subsequence of $(\s_n)_{n\geq 0}$ converges; in particular, we can assume that $\nu'(\s_n\triangle\s_{n+1})\leq\frac{1}{2^n}$ for all $n\geq 0$. In this case, $\overline{\s}\setminus\underline{\s}=\limsup\left(\s_n\triangle\s_{n+1}\right)$ has measure zero by the Borel-Cantelli Lemma. 
\end{proof}

We will refer to $X'$ as the \emph{barycentric subdivision} of $X$. Indeed, if $X$ is the $0$-skeleton of a ${\rm CAT}(0)$ cube complex, $X'$ is the $0$-skeleton of the usual barycentric subdivision. 

A variation on this construction allows us to embed median spaces into connected ones. We define a sequence $(X_n)_{n\geq 0}$ by setting $X_0:=X$ and $X_{n+1}:=X_n'$. These come equipped with compatible isometric embeddings $X_m\hookrightarrow X_n$, for $m<n$; thus, we can consider $\varinjlim X_n$ and call $X''$ its completion. It is a complete median space by Proposition~2.21 in \cite{CDH}.

We have projections $\mscr{H}(X_{n+1})\ra\mscr{H}(X_n)$ with fibres of cardinality one or two; setting, $\mscr{H}'':=\varprojlim\mscr{H}(X_n)$ we still have a projection $p_{\infty}\colon\mscr{H}''\ra\mscr{H}$. It is one-to-one on $p_{\infty}^{-1}(\mscr{H}\setminus\mscr{A})$, while, if $\mf{h}\in\mscr{A}$, the preimage $p_{\infty}^{-1}(\mf{h})$ is canonically in one-to-one correspondence with $\{-,+\}^{\N}$. Let $\mbb{P}$ the standard product probability measure on $\{-,+\}^{\N}$, defined on the $\s$-algebra $\mscr{P}$. For every $E\cu\mscr{H}''$ with $p_{\infty}(E)\setminus\mscr{A}\in\wh{\mscr{B}}$ and $E\cap p_{\infty}^{-1}(\mf{h})\in\mscr{P}$ for all $\mf{h}\in\mscr{A}$, we can set
\[\nu''(E):=\wh{\nu}\left(p(E)\setminus\mscr{A}\right)+\sum_{\mf{h}\in\mscr{A}}\mbb{P}\left(E\cap p_{\infty}^{-1}(\mf{h})\right).\]
This defines a measure on a $\s$-algebra $\mscr{B}''\cu 2^{\mscr{H}''}$. Each point of $X_n$ is represented by a $\mscr{B}''$-measurable ultrafilter on $\mscr{H}''$ and the measure $\nu''$ applied to symmetric differences of such ultrafilters recovers the metric of $X_n$.

Every point of $x\in X''$ is represented by a $\mscr{B}''$-measurable ultrafilter on $\mscr{H}''$. Indeed, we can find points $x_k\in X_k$ converging to $x$ and these are represented by $\mscr{B}''$-measurable ultrafilters $\s_k\cu\mscr{H}''$. The argument in the proof of part~$(1)$ of Proposition~\ref{properties of X'} shows that $\liminf\s_k$ can be completed to a $\mscr{B}''$-measurable ultrafilter $\s\cu\mscr{H}''$ and $\nu''(\s\triangle\s_k)\ra 0$.

Retracing our proofs of Lemmata~\ref{cubes in X'} and~\ref{halfspaces of X'}, it is not hard to show that $\text{Hull}_{X''}(X)=X''$ and that, for every $x\in X''$, there exists a monomorphism $\iota_x^{\infty}\colon [-1,1]^k\hookrightarrow X''$, $1\leq k\leq\text{rank}(X)$, with gate-convex image satisfying $\iota_x^{\infty}(0,...,0)=x$ and $(\iota_x^{\infty})^{-1}(X)=\{-1,1\}^k$. One can also identify all halfspaces of $X''$: they correspond either to elements of $\mscr{H}\setminus\mscr{A}$ or to pairs consisting of an element of $\mscr{A}$ and an element of $\mscr{H}([-1,1])$. In particular, $\text{rank}(X'')=\text{rank}(X)$. Since $(\mscr{H}(X''),\nu'')$ has no atoms, we have obtained:

\begin{cor}\label{X''}
There exists a connected, complete median space $X''$ of the same rank as $X$ such that $X\hookrightarrow X''$ isometrically and every isometry of $X$ extends canonically to an isometry of $X''$, yielding $\text{Isom}~X\hookrightarrow\text{Isom}~X''$.
\end{cor}

\begin{cor}\label{finite orbits}
Let $\G$ be a group of isometries of $X$ and $r:=\text{rank}(X)$. The action $\G\acts X$ has bounded orbits if and only if it has an orbit of cardinality at most $2^r$.
\end{cor}
\begin{proof}
if $\G\acts X$ has bounded orbits, so does the action $\G\acts X''$ arising from Corollary~\ref{X''}. Theorem~\ref{CAT(0) metric} and Cartan's Theorem imply that $\G$ fixes a point $x\in X''$. We showed above that $x$ is the centre of a canonical cube with vertices in $X$; these vertices are at most $2^r$ and provide a finite orbit for $\G\acts X$.
\end{proof}

\section{Groups of elliptic isometries.}\label{elliptic isometries}

Let $X$ be a complete median space of finite rank $r$. This section is devoted to the following result; in the case of ${\rm CAT}(0)$ cube complexes, this is well-known and due to M.~Sageev (see Theorem~5.1 in \cite{Sageev}). 

\begin{thm}\label{Sageev's theorem}
Let $\G$ be a finitely generated group of elliptic isometries of $X$. There exists a finite orbit for $\G\acts X$ and, if $X$ is connected, a global fixed point.
\end{thm}

We start with the following observation, compare the proof of Theorem~5.1 in \cite{Sageev}.

\begin{prop}\label{key point in Sageev}
If $\G$ is finitely generated and acts with unbounded orbits on $X$, there exist $g\in \G$ and $\mf{h}\in\mscr{H}$ such that $g\mf{h}\subsetneq\mf{h}$.
\end{prop}
\begin{proof}
Let $S=\{s_1,...,s_k\}$ be a generating set for $\G$. Choose a basepoint $x\in X$ and some $g\in \G$ with 
\[d(x,gx)>r\cdot\sum_{i=1}^kd(x,s_ix).\] 
We write $g=s_{i_1}...s_{i_n}$ and set $g_j:=s_{i_1}...s_{i_j}$ and $x_j:=g_jx$. We define inductively the points $y_j$, starting with $y_0=x_0=x$ and declaring $y_{j+1}$ to be the projection of $x_{j+1}$ to $I(y_j,gx)$. In particular $(y_j)_{0\leq j\leq n}$ is a geodesic from $x$ to $gx$ and $\mscr{H}(y_j|y_{j+1})\cu\mscr{H}(x_j|x_{j+1})=g_j\mscr{H}(x|s_{i_{j+1}}x)$. The sets $U_j:=g_j^{-1}\mscr{H}(y_j|y_{j+1})$ all lie in $\bigcup_i\mscr{H}(x|s_ix)$ and, since
\[\sum_{j=0}^{n-1}\wh{\nu}(U_j)=\sum_{j=0}^{n-1}d(y_j,y_{j+1})=d(x,gx)>r\cdot\wh{\nu}\left(\bigcup_{i=1}^k\mscr{H}(x|s_ix)\right),\]
there exist $\tau\in\{1,...,k\}$ and a measurable subset $\Om\cu\mscr{H}(x|s_{\tau}x)$ such that $\wh{\nu}(\Om)>0$ and $\Om\cu U_j$ for $r+1$ indices $j_1<j_2<...<j_{r+1}$. Denote by $h_1,...,h_{r+1}$ the corresponding elements of $\G$ such that $h_i\Om\cu\mscr{H}(y_{j_i}|y_{j_i+1})$.

Pick any $\mf{h}\in\Om$. The halfspaces $h_1\mf{h},...,h_{r+1}\mf{h}$ are pairwise distinct, as $(y_j)_{0\leq j\leq n}$ is a geodesic. Moreover, they all lie in $\s_{gx}\setminus\s_x$ and cannot be pairwise transverse. Hence, there exist $1\leq i<j\leq r+1$ such that either $h_i\mf{h}\subsetneq h_j\mf{h}$ or $h_j\mf{h}\subsetneq h_i\mf{h}$. We conclude by setting $g:=h_j^{-1}h_i$ or ${g:=h_i^{-1}h_j}$.
\end{proof}

\begin{lem}\label{nesting vs fixed points}
If $g\in\text{Isom}~X$ and $\mf{h}\in\mscr{H}$ satisfy $g\mf{h}\subsetneq\mf{h}$, then $g$ is not elliptic.
\end{lem}
\begin{proof}
For every $k\geq 1$, we have $g^k\mf{h}\subsetneq\mf{h}$. If $g^k\overline{\mf{h}}\cap\overline{\mf{h}^*}\neq\emptyset$, Proposition~\ref{all about halfspaces} implies that $k+1\leq 2r$; hence $d(g^{2r}\mf{h},\mf{h}^*)>0$. Given $x\in g\mf{h}^*\cap\mf{h}$ and $n\geq 1$:
\[\bigsqcup_{k=1}^{n-1}\mscr{H}(g^{2rk}\mf{h}^*|g^{2r(k+1)}\mf{h})\cu\mscr{H}(x|g^{2rn}x).\]
%union is not disjoint; if we wanted to write a proof of the inequality below, we would have to distinguish the case when $\mf{h}$ is an atom.
Thus $d(x,g^{2rn}x)\geq (n-1)\cdot d(g^{2r}\mf{h},\mf{h}^*)$ and the $\langle g\rangle$-orbit of $x$ is unbounded.
\end{proof}

\begin{proof}[Proof of Theorem~\ref{Sageev's theorem}]
The action $\G\acts X$ has bounded orbits or Proposition~\ref{key point in Sageev} and Lemma~\ref{nesting vs fixed points} would provide a non-elliptic isometry in $\G$. The conclusion follows from Corollary~\ref{finite orbits}.
\end{proof}

In ${\rm CAT}(0)$ cube complexes, Proposition~\ref{key point in Sageev} above implies that the stabiliser $\G_{\mf{h}}$ of $\mf{h}$ is a codimension one subgroup of $\G$. This fails in general for actions on median spaces. 

For instance, surface groups have free actions on real trees \cite{Morgan-Shalen}; in these, all halfspace-stabilisers are trivial. One can still conclude that $\G_{\mf{h}}$ is codimension one as in \cite{Sageev} if, in addition, for every $x,y\in X$ we have $g\mf{h}\in\mscr{H}(x|y)$ only for finitely many left cosets $g\G_{\mf{h}}$.

\section{Stabilisers of points in the Roller boundary.}\label{stab xi}

Let $X$ be a complete median space of finite rank $r$. Let $\xi\in\partial X$ be a point in the Roller boundary; we denote by $\text{Isom}_{\xi}X$ the subgroup of $\text{Isom}~X$ fixing the point $\xi$. The main result of this section will be the following analogue of a result of P.-E.~Caprace (see the appendix to \cite{CFI}).

\begin{thm}\label{stabiliser of xi}
The group $\text{Isom}_{\xi}X$ contains a subgroup $K_{\xi}$ of index at most $r!$ that fits in an exact sequence
\[1\longrightarrow N_{\xi}\longrightarrow K_{\xi} \longrightarrow\R^r.\]
Every finitely generated subgroup of $N_{\xi}$ has an orbit with at most $2^r$ elements.
\end{thm}

In order to prove this, we will have to extend to median spaces part of the machinery developed in \cite{Hagen}.

\begin{defn}\label{UBS}
Given $\xi\in\partial X$, a \emph{chain of halfspaces diverging to $\xi$} is a sequence $(\mf{h}_n)_{n\geq 0}$ of halfspaces of $X$ with $\mf{h}_n\supsetneq\mf{h}_{n+1}$ and $\xi\in\wt{\mf{h}}_n$ for all $n\geq 0$ and such that $d(x,\mf{h}_n)\ra +\infty$ for $x\in X$. A \emph{undirectional boundary set (UBS)} for $\xi\in\partial X$ and $x\in X$ is an inseparable subset $\Om\cu\s_{\xi}\setminus\s_x$ that contains a chain of halfspaces diverging to $\xi$.
\end{defn}

This is an analogue of Definition~3.4 in \cite{Hagen}, except that we consider sets of halfspaces instead of sets of walls. For cube complexes, our definition is a bit more restrictive than Hagen's, since we assume by default that $\Om$ lie in some $\s_{\xi}\setminus\s_x$. This is enough for our purposes and avoids annoying technicalities.

We denote by $\mc{U}(\xi,x)$ the set of all UBS's for $\xi$ and $x$ and we define a relation $\preceq$:
\begin{align*}
\Om_1\preceq\Om_2 & \xLeftrightarrow{\text{def}^{\underline{\text{n}}}}  \sup_{\mf{h}\in\Om_1\setminus\Om_2} d(x,\mf{h})<+\infty,
\end{align*}
which we read as ``$\Om_1$ is \emph{almost contained} in $\Om_2$''. If $\Om_1\preceq\Om_2$ and $\Om_2\preceq\Om_1$, we write $\Om_1\sim\Om_2$ and say that $\Om_1$ and $\Om_2$ are \emph{equivalent}. The relation $\preceq$ descends to a partial order, also denoted $\preceq$, on the set $\overline{\mc{U}}(\xi,x)$ of $\sim$-e\-quiv\-a\-lence classes. We denote the equivalence class associated to the UBS $\Om$ by $[\Om]$. A UBS is said to be \emph{minimal} if it projects to a minimal element of $\overline{\mc{U}}(\xi,x)$. Two UBS's are \emph{almost disjoint} if their intersection is not a UBS.

We will generally forget about the basepoint $x$ and simply write $\left(\overline{\mc{U}}(\xi),\preceq\right)$. Indeed, if $x,y\in X$, we have a canonical isomorphism $\overline{\mc{U}}(\xi,x)\simeq\overline{\mc{U}}(\xi,y)$ given by intersecting UBS's with $\s_{\xi}\setminus\s_x$ or $\s_{\xi}\setminus\s_y$. 

\begin{lem}\label{almost disjoint} 
Let $\Om_1,\Om_2\cu\s_{\xi}\setminus\s_x$ be UBS's. 
\begin{enumerate}
\item If $\Om_1\preceq\Om_2$, we have $\wh{\nu}(\Om_1\setminus\Om_2)<+\infty$. In particular, if $\Om_1$ and $\Om_2$ are equivalent, we have $\wh{\nu}\left(\Om_1\triangle\Om_2\right)<+\infty$.
\item The UBS's $\Om_1,\Om_2$ are almost disjoint if and only if $\Om_1\cap\Om_2$ consists of halfspaces at uniformly bounded distance from $x$.
\end{enumerate}
\end{lem}
\begin{proof}
We first prove part~$(1)$. By Dilworth's Theorem \cite{Dilworth}, we can decompose $\Om_1\setminus\Om_2$ as a disjoint union $\mc{C}_1\sqcup ...\sqcup\mc{C}_k$, where each $\mc{C}_i$ is totally ordered by inclusion and $k\leq r$. If $\Om_1\preceq\Om_2$, Lemma~\ref{intersections of halfspaces} implies that the intersection and union of all halfspaces in $\mc{C}_i$ are halfspaces $\mf{h}_i$ and $\mf{k}_i$, respectively. Thus, $\Om_1\setminus\Om_2$ is contained in the union of the sets $\mscr{H}(\mf{k}_i^*|\mf{h}_i)$, which all have finite measure. 

We now prove part~$(2)$. Since $\Om_1\cap\Om_2$ is inseparable, $\Om_1$ and $\Om_2$ are almost disjoint if and only if $\Om_1\cap\Om_2$ does not contain a chain of halfspaces diverging to $\xi$. By Lemma~\ref{Ramsey}, this is equivalent to $\Om_1\cap\Om_2$ being at uniformly bounded distance from $x$.
\end{proof}

Part~$(1)$ of Lemma~\ref{almost disjoint} is in general not an ``if and only if'' since UBS's can have finite measure. An example appears in the staircase in Figure~3 of \cite{Fioravanti1}, where the UBS is given by the set of vertical halfspaces containing the bottom of the staircase.

\begin{figure}
\centering
\includegraphics[height=2.5in]{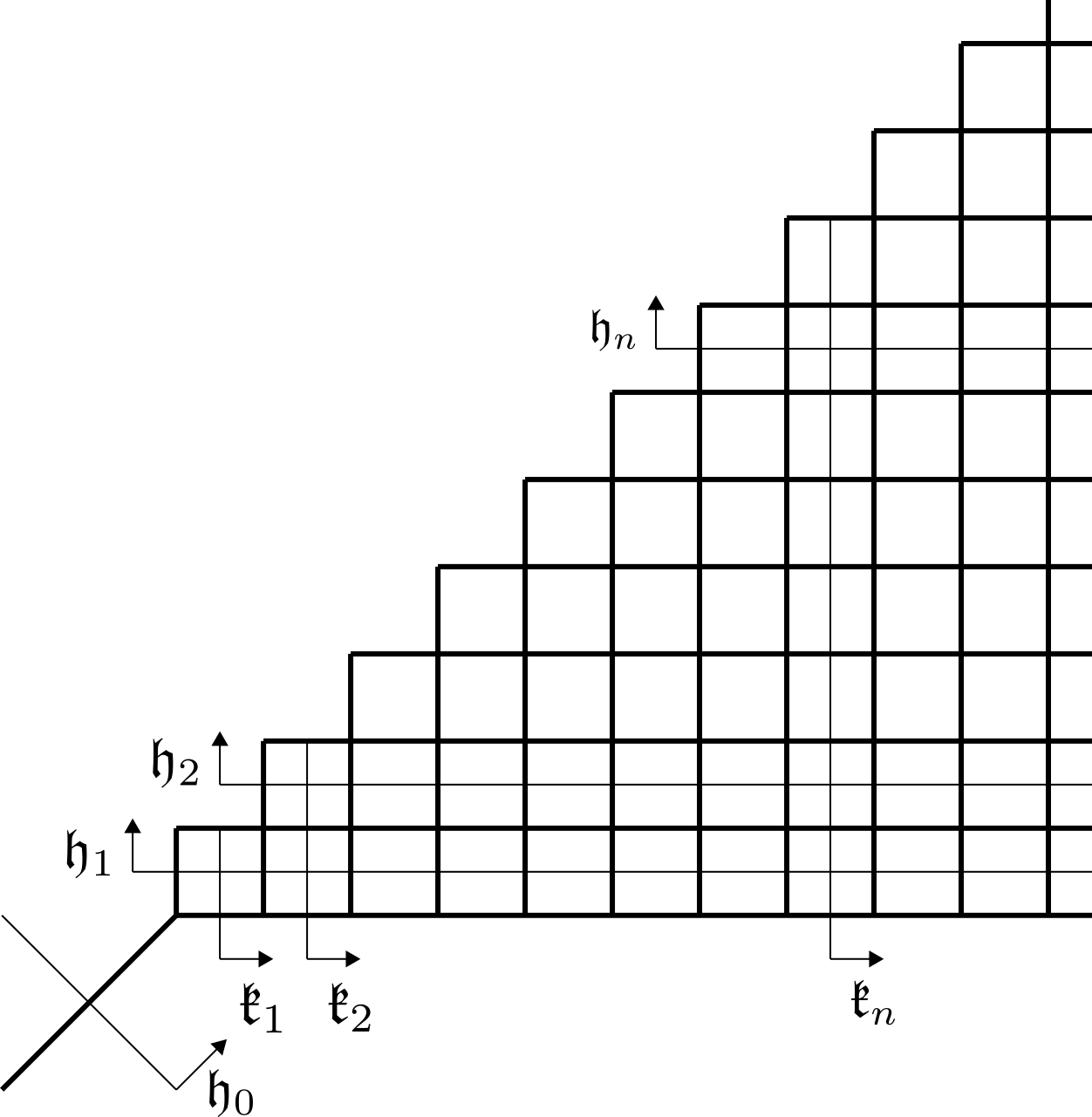}
\caption{}
\label{CSC with a flap}
\end{figure}
The inseparable closure of a chain of halfspaces diverging to $\xi$ is always a UBS and every minimal UBS is equivalent to a UBS of this form. However, not all UBS's of this form are minimal. For instance, consider the ${\rm CAT}(0)$ square complex in Figure~\ref{CSC with a flap}, which is a variation of the usual staircase with a one-dimensional flap. The inseparable closure of $\{\mf{h}_n\}_{n\geq 0}$ is not a minimal UBS, since it also contains all $\mf{k}_n$, while the inseparable closures of $\{\mf{h}_n\}_{n\geq 1}$ and $\{\mf{k}_n\}_{n\geq 1}$ are minimal.

\begin{lem}\label{symmetric almost-transversality}
Let $(\mf{h}_m)_{m\geq 0}$ and $(\mf{k}_n)_{n\geq 0}$ be chains of halfspaces in $\s_{\xi}\setminus\s_x$ that diverge to $\xi$. Suppose that no $\mf{h}_m$ lies in the inseparable closure of $\{\mf{k}_n\}_{n\geq 0}$. Either almost every $\mf{k}_n$ is transverse to almost every $\mf{h}_m$, or almost every $\mf{h}_m$ is transverse to almost every $\mf{k}_n$.
\end{lem}
\begin{proof} 
We first suppose that there exist $\overline n,\overline m\geq 0$ such that $\mf{h}_{\overline m}\cu\mf{k}_{\overline n}$; without loss of generality, $\overline m=\overline n=0$. An inclusion of the type $\mf{k}_n\cu\mf{h}_m$ can never happen, or we would have $\mf{k}_n\cu\mf{h}_m\cu\mf{h}_0\cu\mf{k}_0$ and $\mf{h}_0$ would lie in the inseparable closure of $\{\mf{k}_n\}_{n\geq 0}$. Since the sequence $(\mf{k}_n)_{n\geq 0}$ diverges, for every $m$ there exists $n(m)$ such that $\mf{k}_n$ does not contain $\mf{h}_m$ for $n\geq n(m)$; thus, for $n\geq n(m)$, the halfspaces $\mf{k}_n$ and $\mf{h}_m$ are transverse. 

Now suppose instead that an inclusion of the form $\mf{h}_m\cu\mf{k}_n$ never happens. For every $n$, there exists $m(n)$ such that $\mf{k}_n$ is not contained in $\mf{h}_m$ for ${m\geq m(n)}$; thus, for $m\geq m(n)$, the halfspaces $\mf{k}_n$ and $\mf{h}_m$ are transverse.
\end{proof}

Following \cite{corrigendum}, we construct a directed graph $\mc{G}(\xi)$ as follows. Vertices of $\mc{G}(\xi)$ correspond to minimal elements of $\left(\overline{\mc{U}}(\xi),\preceq\right)$. Given diverging chains $(\mf{h}_m)_{m\geq 0}$ and $(\mf{k}_n)_{n\geq 0}$ in minimal UBS's $\Om$ and $\Om'$, respectively, we draw an oriented edge from $[\Om]$ to $[\Om']$ if almost every $\mf{h}_m$ is transverse to almost every $\mf{k}_n$, but the same does not happen if we exchange $(\mf{h}_m)_{m\geq 0}$ and $(\mf{k}_n)_{n\geq 0}$. This does not depend on which diverging chains we pick, as $\Om$, $\Om'$ are minimal.

By Lemma~\ref{symmetric almost-transversality}, the vertices corresponding to $\Om$ and $\Om'$ are not joined by any edge if and only if almost every $\mf{h}_m$ is transverse to almost every $\mf{k}_n$ and almost every $\mf{k}_n$ is transverse to almost every $\mf{h}_m$. It is clear that there are no directed cycles of length $2$ in $\mc{G}(\xi)$.

\begin{lem}\label{no directed cycles}
If there is a directed path from $[\Om]$ to $[\Xi]$, there is an oriented edge from $[\Om]$ to $[\Xi]$. In particular, $\mc{G}(\xi)$ contains no directed cycles.
\end{lem}
\begin{proof}
It suffices to prove that, if there is an edge from $[\Om]$ to $[\Om']$ and from $[\Om']$ to $[\Om'']$, there is also an edge from $[\Om]$ to $[\Om'']$. Pick diverging chains $(\mf{h}_n)_{n\geq 0}$, $(\mf{h}'_n)_{n\geq 0}$, $(\mf{h}''_n)_{n\geq 0}$ in $\Om$, $\Om'$, $\Om''$, respectively. By hypothesis, there are infinitely many $\mf{h}_k''$ that are not transverse to almost every $\mf{h}'_j$; thus, for every $k$ there exists $j$ such that $\mf{h}''_k\supseteq\mf{h}'_j$. A similar argument works for $\Om$ and $\Om'$; hence, for every $k$ there exist $i,j$ such that $\mf{h}''_k\supseteq\mf{h}'_j\supseteq\mf{h}_i$. If no oriented edge from $[\Om]$ to $[\Om'']$ existed, almost every $\mf{h}''_k$ would be transverse to almost every $\mf{h}_i$ and this would contradict the previous statement.
\end{proof}

\begin{rmk}\label{graph for any chains}
A graph like $\mc{G}(\xi)$ above can be constructed whenever we have a family of diverging chains $(\mf{h}_n^i)_{n\geq 0}$, $i\in I$, with the property that, if $i\neq j$, either no $\mf{h}_m^i$ lies in the inseparable closure of $\{\mf{h}_n^j\}_{n\geq 0}$ or vice versa. Lemma~\ref{no directed cycles} and part~$(1)$ of Proposition~\ref{main prop on UBS's} below still hold in this context.
\end{rmk}

We say that a collection of vertices $\mc{V}\cu\mc{G}(\xi)^{(0)}$ is \emph{inseparable} if, for every $v,w\in\mc{V}$, all the vertices on the directed paths from $v$ to $w$ also lie in $\mc{V}$. The following extends Lemma~3.7 and Theorem~3.10 in \cite{Hagen}.

\begin{prop}\label{main prop on UBS's}
\begin{enumerate}
\item The graph $\mc{G}(\xi)$ has at most $r$ vertices.
\item For every UBS $\Om$ there exists a minimal UBS $\Om'\preceq\Om$. If $\Om$ is the inseparable closure of a diverging chain $(\mf{h}_n)_{n\geq 0}$, we can take $\Om'$ to be the inseparable closure of $(\mf{h}_n)_{n\geq N}$, for some $N\geq 0$.
\item Given a UBS $\Om$ and a set $\{\Om_1,...,\Om_k\}$ of representatives of all equivalence classes of minimal UBS's almost contained in $\Om$, we have
\[\sup_{\mf{h}\in\Om\triangle\left(\Om_1\cup ... \cup\Om_k\right)} d(x,\mf{h})<+\infty.\]
\item There is an isomorphism of posets between $\left(\overline{\mc{U}}(\xi),\preceq\right)$ and the collection of inseparable subsets of $\mc{G}(\xi)^{(0)}$, ordered by inclusion. It is given by associating to $[\Om]$ the set $\{[\Om_1],...,[\Om_k]\}$ of minimal equivalence classes of UBS's almost contained in $\Om$.
\end{enumerate}
\end{prop}
\begin{proof}
To prove part~$(1)$, we show that every finite subset $\mc{V}\cu\mc{G}(\xi)^{(0)}$ satisfies ${\#\mc{V}\leq r}$. More precisely, we prove by induction on $k$ that, if $\Om_1,...,\Om_k$ are UBS's representing the elements of $\mc{V}$, we can find pairwise-transverse halfspaces $\mf{h}_i\in\Om_i$. The case $k=1$ is trivial; suppose $k\geq 2$. By Lemma~\ref{no directed cycles} we can assume, up to reordering the $\Om_i$, that there is no edge from $[\Om_i]$, $i\leq k-1$, to $[\Om_k]$. There exist $\mf{h}\in\Om_k$ and diverging chains $\{\mf{k}_n^i\}_{n\geq 0}\cu\Om_i$, $i\leq k-1$, that are transverse to $\mf{h}$; in particular, $\mf{h}$ is transverse to every element in the inseparable closure of $\{\mf{k}_n^i\}_{n\geq 0}$, for $i\leq k-1$. By the inductive hypothesis, we can find $\mf{h}_i$ in the inseparable closure of $\{\mf{k}_n^i\}_{n\geq 0}$ so that $\mf{h}_1,...,\mf{h}_{k-1}$ are pairwise transverse. Hence $\mf{h},\mf{h}_1,...,\mf{h}_{k-1}$ are pairwise transverse.

We now prove part~$(2)$. If $\Om_1\prec...\prec\Om_k$ is a chain of non-equivalent UBS's, we have $k\leq r$. Indeed, we can consider diverging chains in $\Om_1$ and in $\Om_i\setminus\Om_{i-1}$ for $2\leq i\leq k$, which exist by Lemma~\ref{Ramsey}, and appeal to Remark~\ref{graph for any chains} to produce $k$ pairwise transverse halfspaces. 

This implies the existence of minimal UBS's almost contained in any UBS. It also shows that, for every diverging chain $(\mf{h}_n)_{n\geq 0}$, there exists $N\geq 0$ such that the inseparable closures $\Om_M$ of $(\mf{h}_n)_{n\geq M}$ are all equivalent for $M\geq N$. In particular, every diverging chain in $\Om_N$ has a cofinite subchain that is contained in $\Om_M$, if ${M\geq N}$. By Lemma~\ref{symmetric almost-transversality}, the UBS $\Om_N$ is equivalent to the inseparable closure of any diverging chain it contains, i.e.~$\Om_N$ is minimal. 

Regarding part~$(3)$, it is clear that the supremum over $\left(\Om_1\cup ... \cup\Om_k\right)\setminus\Om$ is finite. If the supremum over $\Om\setminus\left(\Om_1\cup ... \cup\Om_k\right)$ were infinite, Lemma~\ref{Ramsey} and part~$(2)$ would provide a diverging chain in $\Om\setminus\left(\Om_1\cup ... \cup\Om_k\right)$ whose inseparable closure $\Om'$ is a minimal UBS. Thus $\Om'\preceq\Om$, but $\Om'\not\sim\Om_i$ for all $i$, a contradiction. % here the second part of part~2 is really important, even though it might not seem so at first sight 

Finally, we prove part~$(4)$. The map $[\Om]\mapsto\{[\Om_1],...,[\Om_k]\}$ is an injective morphism of posets by part~$(3)$. The collection $\{[\Om_1],...,[\Om_k]\}$ is inseparable since the inseparable closure of $(\Om_i\cap\Om)\cup(\Om_j\cap\Om)$ contains all minimal UBS's corresponding to vertices on directed paths from $[\Om_i]$ to $[\Om_j]$ and vice versa; this follows for instance from the proof of Lemma~\ref{no directed cycles}.

Given an inseparable collection $\{[\Om_1],...,[\Om_k]\}$, we construct a UBS $\Om$ such that these are precisely the equivalence classes of minimal UBS's almost contained in $\Om$. Let $(\mf{h}_n^i)_{n\geq 0}$ be a diverging chain in $\Om_i$, for every $i$, and denote by $\Om^N$ the inseparable closure of ${\{\mf{h}_n^1\}_{n\geq N}\cup...\cup\{\mf{h}_n^k\}_{n\geq N}}$. If $N$ is large enough, every minimal UBS almost contained in $\Om^N$ is equivalent to one of the $\Om_i$. Otherwise, by part~$(1)$, we would be able to find a diverging chain $\{\mf{h}_n\}_{n\geq 0}$ such that its inseparable closure $\Xi$ is not equivalent to any of the $\Om_i$ and $\mf{h}_{a_n}^j\cu\mf{h}_n\cu\mf{h}_{b_n}^k$, for some $j,k$, with $a_n,b_n\ra+\infty$. 
% indeed, by part~1 there exists $\Xi$ almost contained in all $\Om^N$
This implies that $\Xi$ lies on a directed path from $[\Om_j]$ to $[\Om_k]$ and contradicts inseparability of the collection $\{[\Om_1],...,[\Om_k]\}$.
\end{proof}

Let $\Om\cu\s_{\xi}\setminus\s_x$ be a UBS and let $K_{\Om}\leq\text{Isom}_{\xi}X$ be the subgroup of isometries that preserve the equivalence class $[\Om]$. We introduce the map 
\begin{align*}
\chi_{\Om}\colon K_{\Om} \longrightarrow & \R \nonumber \\
g\longmapsto & \wh{\nu}\left(g^{-1}\Om\setminus\Om\right)- \wh{\nu}\left(\Om\setminus g^{-1}\Om\right) .
\end{align*}
Note that the definition makes sense due to part~$(1)$ of Lemma~\ref{almost disjoint}. We will refer to $\chi_{\Om}$ as the \emph{transfer character} associated to $[\Om]$. The terminology is motivated by the following analogue of (part of) Proposition~4.H.1 in \cite{Cor}.

\begin{lem}
The map $\chi_{\Om}$ is a homomorphism. Moreover, 
\[\chi_{\Om}(g)=\wh{\nu}\left(g^{-1}\Xi\setminus\Xi\right)- \wh{\nu}\left(\Xi\setminus g^{-1}\Xi\right) \]
for every $g\in K_{\Om}$ and every measurable subset $\Xi\cu\mscr{H}$ with $\wh{\nu}(\Om\triangle\Xi)<+\infty$.
\end{lem}
\begin{proof}
Given $A\cu\mscr{H}$, let $\mathds{1}_A$ denote its characteristic function. Given any function $f\colon\mscr{H}\ra\R$ and $g\in\text{Isom}~X$, we employ the standard notation $(g\cdot f)(x)=f(g^{-1}x)$. Observe that $\chi_{\Om}(g)=\int(g^{-1}\cdot\mathds{1}_{\Om}-\mathds{1}_{\Om})d\wh{\nu}$.

Proving that $\chi_{\Om}$ is a homomorphism amounts to the observation that 
\[\chi_{\Om}(gh)=\int\left(h^{-1}g^{-1}\cdot\mathds{1}_{\Om}-\mathds{1}_{\Om}\right)d\wh{\nu}=\]
\[=\int\left(h^{-1}g^{-1}\cdot\mathds{1}_{\Om}-h^{-1}\cdot\mathds{1}_{\Om}\right)d\wh{\nu}+\int\left(h^{-1}\cdot\mathds{1}_{\Om}-\mathds{1}_{\Om}\right)d\wh{\nu}=\]
\[=\int\left(g^{-1}\cdot\mathds{1}_{\Om}-\mathds{1}_{\Om}\right)d\wh{\nu}+\int\left(h^{-1}\cdot\mathds{1}_{\Om}-\mathds{1}_{\Om}\right)d\wh{\nu}=\chi_{\Om}(g)+\chi_{\Om}(h),\]
since $g,h\in K_{\Om}$ and, in particular, $h$ preserves the measure $\wh{\nu}$. Regarding the second statement, it suffices to consider the case $\Xi=\Om\sqcup F$, with $\wh{\nu}(F)<+\infty$. We then have
\[\int\left(g^{-1}\cdot\mathds{1}_{\Xi}-\mathds{1}_{\Xi}\right)d\wh{\nu}=\int\left(g^{-1}\cdot\mathds{1}_{\Om}-\mathds{1}_{\Om}\right)d\wh{\nu}+\int\left(g^{-1}\cdot\mathds{1}_{F}-\mathds{1}_{F}\right)d\wh{\nu}\]
and
\[\int\left(g^{-1}\cdot\mathds{1}_{F}-\mathds{1}_{F}\right)d\wh{\nu}=\int (g^{-1}\cdot\mathds{1}_{F})d\wh{\nu}-\int\mathds{1}_{F}d\wh{\nu}=0.\]
\end{proof}

Thus, the transfer character $\chi_{\Om}$ only depends on the equivalence class $[\Om]$ of the UBS $\Om$. If $\{\Om_1,...,\Om_k\}$ is a set of representatives of all equivalence classes of minimal UBS's almost contained in $\Om$, we have $\chi_{\Om}=\chi_{\Om_1}+...+\chi_{\Om_k}$ by part~$(3)$ of Proposition~\ref{main prop on UBS's}.

Now, let $\Om_1,...,\Om_k$ be UBS's representing all minimal elements of $\overline{\mc{U}}(\xi)$. The group $\text{Isom}_{\xi}X$ permutes the equivalence classes of the $\Om_i$ and a subgroup $K_{\xi}\leq\text{Isom}_{\xi}X$ of index at most $k!\leq r!$ preserves them all. Note that, by part~$(4)$ of Proposition~\ref{main prop on UBS's}, this is precisely the kernel of the action of $\text{Isom}_{\xi}X$ on $\overline{\mc{U}}(\xi)$. We define a homomorphism $\chi_{\xi}:=(\chi_{\Om_1},...,\chi_{\Om_k})\colon K_{\xi}\ra \R^k$.

\begin{prop}\label{kernel of chi}
Every finitely generated subgroup $\G\leq\ker\chi_{\xi}$ has an orbit in $X$ with at most $2^r$ elements.
\end{prop}
\begin{proof}
If $\G$ did not have an orbit with at most $2^r$ elements, all orbits would be unbounded by Corollary~\ref{finite orbits} and Proposition~\ref{key point in Sageev} would provide ${\mf{h}\in\mscr{H}}$ and $g\in\G$ with $g\mf{h}\subsetneq\mf{h}$; hence ${d(g^r\mf{h},\mf{h}^*)>0}$ by Proposition~\ref{all about halfspaces}. If $\xi\in\wt{\mf{h}}^*$, we replace $\mf{h}$ with $\mf{h}^*$ and $g$ with $g^{-1}$. Now $(g^{nr}\mf{h})_{n\geq 0}$ is a sequence of halfspaces diverging to $\xi$ and, by part~$(2)$ of Proposition~\ref{main prop on UBS's}, the inseparable closure $\Om^N$ of $\{g^{nr}\mf{h}\}_{n\geq N}$ is a minimal UBS if $N$ is large enough. Thus, $\Om^N\sim\Om_i$ for some $i$ and we have ${0=\chi_{\Om_i}(g)=\chi_{\Om^N}(g)}$. We obtain a contradiction by observing that
\[r\cdot\chi_{\Om^N}(g)=\chi_{\Om^N}(g^r)\geq\wh{\nu}\left(\mscr{H}(\mf{h}^*|g^r\mf{h})\setminus\{g^r\mf{h}\}\right)> 0.\]
\end{proof}

Theorem~\ref{stabiliser of xi} immediately follows from Proposition~\ref{kernel of chi}. Relying on Theorem~E, we can already provide a proof of Theorem~A. Theorem~E will be proved in Section~\ref{facing triples}.

\begin{proof}[Proof of Theorem~A]
By hypothesis, $\G$ does not have any nonabelian free subgroups; thus, the action $\G\acts X$ is Roller elementary by Theorem~E. Theorem~\ref{stabiliser of xi} yields a finite-index subgroup $\G_0\leq\G$ and a normal subgroup $N\lhd\G_0$ such that $\G_0/N$ is abelian and every finitely generated subgroup of $N$ has an orbit with at most $2^r$ elements. 

$(1)$ If $\G\acts X$ is free, every finitely generated subgroup of $N$ has at most $2^r$ elements. In particular, $N$ is finitely generated and $\#N\leq 2^r$; this shows that $\G_0$ is finite-by-abelian. Finitely generated finite-by-abelian groups are virtually abelian, see e.g.~Lemma~II.7.9 in \cite{BH}; thus, we can conclude that $\G$ is virtually abelian as soon as $\G$ is finitely generated. Alternatively, if $X$ is connected, $N$ must be trivial by Theorem~\ref{CAT(0) metric} and Cartan's fixed point theorem (Theorem~3.74 in \cite{DK}); thus, we obtain that $\G$ is virtually abelian also in this case.

$(2)$ Assume now instead that $\G\acts X$ is proper. Every finitely generated subgroup of $N$ acts with bounded orbits and must therefore be finite. Hence $N$ is locally finite in this case, as required.

$(3)$ Suppose finally that $\G$ acts with amenable point stabilisers. Every finitely generated subgroup of $N$ is virtually a point stabiliser, hence $N$ is locally amenable. Since direct limits preserve amenability and $N$ is the direct limit of its finitely generated subgroups, we conclude that $N$ is amenable. Thus, $\G_0$ and $\G$ are amenable as well.
\end{proof}

\section{Caprace-Sageev machinery.}\label{CS-like machinery}

Let $X$ be a complete median space of finite rank $r$. The goal of this section is extending to median spaces Theorem~4.1 and Proposition~5.1 from \cite{CS}. 

Our techniques provide a different approach also in the case of ${\rm CAT}(0)$ cube complexes, as we use Roller boundaries instead of visual boundaries. This strategy of proof was suggested to us by T. Fern\'os.

Let $\G$ be a group of isometries of $X$. We say that $g\in \G$ \emph{flips} $\mf{h}\in\mscr{H}$ if ${d\left(g\mf{h}^*,\mf{h}^*\right)>0}$ and $g\mf{h}^*\neq\mf{h}$. The halfspace $\mf{h}$ is \emph{$\G$-flippable} if some $g\in \G$ flips it.

\begin{thm}\label{flipping}
Suppose $\G$ acts without wall inversions. For every thick halfspace, exactly one of the following happens:
\begin{enumerate}
\item $\mf{h}$ is $\G$-flippable;
\item the closure of $\wt{\mf{h}}^*$ in $\overline X$ contains a proper, closed, convex, $\G$-invariant subset $C\cu\overline X$.
\end{enumerate}
\end{thm}
\begin{proof}
If $\mf{h}$ is $\G$-flippable, $\overline{\mf{h}^*}$ and $g\overline{\mf{h}^*}$ are disjoint subsets of $X$; let $(x,x')$ be a pair of gates and $I:=I(x,x')$. Observe that $\pi_I$ maps the closure of $\wt{\mf{h}}^*$ to $x$ and the closure of $g\wt{\mf{h}}^*$ to $x'$. Hence, any wall of $X$ separating $x$ and $x'$ induces a wall of $\overline X$ separating the closures of $\wt{\mf{h}}^*$ and $g\wt{\mf{h}}^*$. Thus, options~$(1)$ and~$(2)$ are mutually exclusive. If $(1)$ does not hold, we have $g\overline{\mf{h}^*}\cap\overline{\mf{h}^*}\neq\emptyset$ for every $g\in\G$, since the action has no wall inversions. Helly's Theorem implies that the closures of the sets $g\wt{\mf{h}}^*$, $g\in\G$, have the finite intersection property and, since $\overline X$ is compact, their intersection $C$ is nonempty. It is closed, convex and $\G$-invariant; since $\mf{h}$ is thick, we have $C\neq\overline X$. 
\end{proof}

The thickness assumption in Theorem~\ref{flipping} is necessary. Consider the real tree obtained from the ray $[0,+\infty)$ by attaching a real line $\ell_n$ to the point $\frac{1}{n}$ for every $n\geq 1$. Complete this to a real tree $T$ so that there exist isometries $g_n$ with axes $\ell_n$; let $\G$ be the group generated by these. The minimal subtree for $\G$ contains all the lines $\ell_n$; let $X$ be its closure in $T$. The action $\G\acts X$ does not preserve any proper, closed, convex subset of $\overline X$, but the singleton $\{0\}$ inside the original ray is a halfspace that is not flipped by $\G$.

We remark that any action on a connected median space is automatically without wall inversions by Proposition~\ref{all about halfspaces}. When $X$ is connected, we denote by $\partial_{\infty}X$ the visual boundary of the ${\rm CAT}(0)$ space arising from Theorem~\ref{CAT(0) metric}. If no proper, closed, convex subset of $X$ is $\G$-invariant, the following describes the only obstruction to flippability of halfspaces.

\begin{prop}\label{visual boundary vs Roller boundary}
If $X$ is connected, there exists a closed, convex, $\G$-invariant subset $C\cu\partial X$ if and only if $\G$ fixes a point of $\partial_{\infty}X$.
\end{prop}
\begin{proof}
Suppose $C\cu\partial X$ is closed, convex and $\G$-invariant. Lemma~2.6 in \cite{Fioravanti1} implies that $C$ is gate-convex, hence the set $\s_C:=\{\mf{h}\in\mscr{H}\mid C\cu\wt{\mf{h}}\}$ is nonempty as it contains all halfspaces separating $x\in X$ and $\pi_C(x)$. By Theorem~\ref{flipping}, any $\mf{h}\in\s_C^*$ is not $\G$-flippable; thus $\{g\overline{\mf{h}}\mid\mf{h}\in\s_C,~g\in\G\}$ is a collection of subsets of $X$ with the finite intersection property. These subsets are convex also with respect to the ${\rm CAT}(0)$ metric and their intersection is empty. The topological dimension of every compact subset of $X$ is bounded above by the rank of $X$, see Lemma~2.10 in \cite{Fioravanti1} and Theorem~2.2, Lemma~7.6 in \cite{Bow1}; thus, the geometric and telescopic dimensions (see \cite{CL}) of the ${\rm CAT}(0)$ metric are at most $r$. The existence of a fixed point in $\partial_{\infty}X$ now follows from Proposition~3.6 in \cite{CS}.

Conversely, suppose $\zeta\in\partial_{\infty}X$ is fixed by $\G$. The intersection of a halfspace of $X$ and a ray for the ${\rm CAT}(0)$ metric is either empty, bounded or a subray. Hence, given $x\in X$, the subset $\s(x,\zeta)\cu\mscr{H}$ of halfspaces intersecting the ray $x\zeta$ in a subray is an ultrafilter; it represents a point $\xi(x,\zeta)\in\overline X$. If $y\in X$ is another point and $x_n,y_n$ are points diverging along the rays $x\zeta$ and $y\zeta$, we have $\s(x,\zeta)\triangle\s(y,\zeta)\cu\liminf\left(\mscr{H}(x_n|y_n)\cup\mscr{H}(y_n|x_n)\right)$. For every $n$, the points $x_n$ and $y_n$ are at most as far apart as $x$ and $y$ in the ${\rm CAT}(0)$ metric; since the latter is bi-Lipschitz equivalent to the median metric on $X$, we conclude that $\wh{\nu}\left(\s(x,\zeta)\triangle\s(y,\zeta)\right)<+\infty$. Thus, $\xi(x,\zeta)$ and $\xi(y,\zeta)$ lie in the same component $Z\cu\overline X$, which is $\G$-invariant. Moreover, $\xi(x,\zeta)\not\in X$ as $\mscr{H}(x|z)\cu\s(x,\zeta)$ for every $z$ on the ray $x\zeta$; hence $Z\cu\partial X$. Finally, it is easy to show that $\overline Z\cu\partial X$.
\end{proof}

We are interested in studying actions where every thick halfspace is flippable, see Corollary~\ref{double skewering} below. To this end, we introduce the following notions of non-elementarity.

\begin{defn}\label{elementarity notion}
We say that the action $\G\acts X$ is:
\begin{itemize} 
\item \emph{Roller nonelementary} if $\G$ has no finite orbit in $\overline X$; 
\item \emph{Roller minimal} if $X$ is not a single point and $\G$ does not preserve any proper, closed, convex subset of the Roller compactification $\overline X$.
\item \emph{essential} if $\G$ does not preserve any proper, closed, convex subset of the median space $X$.
\end{itemize}
\end{defn}

The action of $\G$ is Roller elementary if and only if a finite-index subgroup of $\G$ fixes a point of $\overline X$; thus, Roller nonelementarity passes to finite index subgroups. This fails for Roller minimality. For instance, consider the action of $\G=\Z^2\rtimes\Z/4\Z$ on the the standard cubulation of $\R^2$; the action of $H:=\Z^2$ is by translations, whereas $\Z/4\Z$ rotates around the origin. The action of $\G$ is Roller minimal, but $H$ has four fixed points in the Roller compactification.

The same example shows that Roller minimal actions might not be Roller nonelementary. Roller nonelementary actions need not be Roller minimal either: Let $T$ be the Cayley graph of a nonabelian free group $F$ and consider the product action of $F\x\Z$ on $T\x\R$. It is Roller nonelementary but leaves invariant two components of the Roller boundary, both isomorphic to $T$.

By Proposition~\ref{visual boundary vs Roller boundary}, an essential action $\G\acts X$ is Roller minimal if and only if no point of the visual boundary $\partial_{\infty}X$ is fixed by $\G$. In particular, an essential action with no finite orbits in $\partial_{\infty}X$ is always Roller minimal and Roller nonelementary. 

The following is immediate from Theorem~\ref{flipping} and the proof of the Double Skewering Lemma in the introduction of \cite{CS}.

\begin{cor}\label{double skewering}
If $\G\acts X$ is Roller minimal and without wall inversions, every thick halfspace is $\G$-flippable. Moreover, if $\mf{h}\cu\mf{k}$ are thick halfspaces, there exists $g\in\G$ such that $g\mf{k}\subsetneq\mf{h}\cu\mf{k}$ and $d(g\mf{k},\mf{h}^*)>0$.
\end{cor}

One can usually reduce to studying a Roller minimal action by appealing to the following result. 

\begin{prop}\label{Roller elementary vs strongly so}
Either $\G\acts\overline X$ fixes a point or there exist a $\G$-invariant component $Z\cu\overline X$ and a $\G$-invariant, closed, convex subset ${C\cu Z}$ such that $\G\acts C$ is Roller minimal.
\end{prop}
\begin{proof}
Let $K\cu\overline X$ be a minimal, nonempty, closed, $\G$-invariant, convex subset; it exists by Zorn's Lemma. Corollary~4.31 in \cite{Fioravanti1} provides a component $Z\cu\overline X$ of maximal rank among those that intersect $K$. Since $Z$ must be $\G$-invariant, we have $\overline Z\cap K= K$ by the minimality of $K$, i.e.~$K\cu\overline Z$. The set $C:=K\cap Z$ is nonempty, convex, $\G$-invariant and closed in $Z$, since the inclusion $Z\hookrightarrow\overline X$ is continuous. By minimality of $K$, we have $K=\overline C$ and the latter can be identified with the Roller compactification of $C$ (see Lemma~4.8 in \cite{Fioravanti1}). We conclude that either $\G\acts C$ is Roller minimal or $C$ is a single point.
\end{proof}

\begin{cor}\label{Roller elementary vs strongly so 2}
If $\G\acts X$ is Roller nonelementary, there exist a $\G$-in\-variant component $Z\cu\overline X$ and a $\G$-invariant, closed, convex subset $C\cu Z$ such that $\G\acts C$ is Roller minimal and Roller nonelementary.
\end{cor}

We remark that the following is immediate from part~$(5)$ of Proposition~\ref{properties of X'}:

\begin{lem}\label{RNE for X'}
The action $\G\acts X$ is Roller elementary if and only if the action $\G\acts X'$ is.
\end{lem}

We now proceed to obtain an analogue of Proposition~5.1 from \cite{CS}, namely Theorem~\ref{strong separation} below. We say that $\mf{h},\mf{k}\in\mscr{H}$ are \emph{strongly separated} if $\overline{\mf{h}}\cap\overline{\mf{k}}=\emptyset$ and no $\mf{j}\in\mscr{H}$ is transverse to both $\mf{h}$ and $\mf{k}$.

\begin{lem}\label{slight reformulation of SS}
Halfspaces with disjoint closures are strongly separated if and only if no thick halfspace is transverse to both.
\end{lem}
\begin{proof}
Suppose that $\overline{\mf{h}_1}\cap\overline{\mf{h}_2}=\emptyset$ and a nowhere-dense halfspace $\mf{k}$ is transverse to both $\mf{h}_i$. Pick points $y_i\in\mf{h}_i\cap\mf{k}^*$ and observe that $I:=I(y_1,y_2)\cu\mf{k}^*$; since $\mf{k}$ is closed by Proposition~\ref{all about halfspaces}, we have $d(I,\mf{k})>0$. Thus, if $(x_1,x_2)$ is a pair of gates for $(I,\mf{k})$, the set $\mscr{H}(x_1|x_2)$ has positive measure and it contains a thick halfspace $\mf{k}'$. It is easy to see that $\mf{k}'$ is transverse $\mf{h}_1$ and $\mf{h}_2$.
\end{proof}

\begin{thm}\label{strong separation}
If $\G\acts X$ is Roller minimal and without wall inversions, the following are equivalent:
\begin{enumerate}
\item $X$ is irreducible;
\item there exists a pair of strongly separated halfspaces;
\item for every $\mf{h}\in\mscr{H}\setminus\mscr{H}^{\times}$, there exist halfspaces $\mf{h}'\cu\mf{h}\cu\mf{h}''$ so that $\mf{h}'$ and $\mf{h}''^*$ are thick and strongly separated.
\end{enumerate}
\end{thm}

{\small {\bf Note:} The proof of Theorem~\ref{strong separation} closely follows the proof of Proposition~5.1 in \cite{CS}. Only minor changes are required to address the pathologies that may arise in non-cubical median spaces. We only give what we feel are the relevant bits, referring the reader to Caprace and Sageev's paper for complete proofs. Their arguments can be repeated word by word when we omit them. We advise the reader to make themselves familiar with the entire Section~5 of \cite{CS} before attempting to read what follows.} 

\begin{proof}[Proof of Theorem~\ref{strong separation}]
Observe that $(3)$ clearly implies $(2)$ and $(1)$ follows from $(2)$ using Proposition~\ref{products}. We are left to prove that $(1)$ implies $(3)$. Suppose for the sake of contradiction that, for some $\mf{h}\in\mscr{H}\setminus\mscr{H}^{\times}$, we cannot find $\mf{h}'$ and $\mf{h}''$. We reach a contradiction as in the last paragraph of the proof of Proposition~5.1 in \cite{CS}, once we construct sequences $(\mf{h}_n')_{n\geq 0}$, $(\mf{h}_n'')_{n\geq 0}$ and $(\mf{k}_n)_{n\geq 0}$ of thick halfspaces such that
\begin{enumerate}
\item $\mf{k}_n$ is transverse to $\mf{h}_{n-1}'$ and $\mf{h}_{n-1}''$ for $n\geq 1$;
\item $\mf{k}_n\in\mscr{H}(\mf{h}_n''^*|\mf{h}_n')$ for $n\geq 0$;
\item $\mf{h}_n'\subsetneq\mf{h}_{n-1}'\subsetneq\mf{h}\subsetneq\mf{h}_{n-1}''\subsetneq\mf{h}_n''$ for $n\geq 1$.
\end{enumerate}
By Corollary~\ref{double skewering}, we can find $g\in\G$ such that $g^{-1}\mf{h}\subsetneq\mf{h}\subsetneq g\mf{h}$ and we set $\mf{h}_0':=g^{-1}\mf{h}$ and $\mf{h}_0'':=g\mf{h}$. Now suppose that we have defined $\mf{h}_n'$, $\mf{h}_n''$ and $\mf{k}_{n-1}$. Corollary~\ref{double skewering} yields $g'\in\G$ with $\mf{h}_n'\subsetneq\mf{h}\subsetneq\mf{h}_n''\subsetneq g'\mf{h}_n'\subsetneq g'\mf{h}_n''$ and $\overline{\mf{h}_n''}\cap g'\overline{\mf{h}_n'^*}=\emptyset$. By hypothesis, $\mf{h}_n'$ and $g'\mf{h}_n''^*$ are not strongly separated, but they have disjoint closures; thus there exists $\mf{k}$ transverse to $\mf{h}_n'$ and $g'\mf{h}_n''$. By Lemma~\ref{slight reformulation of SS}, we can assume that $\mf{k}$ is thick. 

The construction of the sequences can be concluded as in \cite{CS} once we obtain analogues of their Lemmata~5.2,~5.3 and~5.4. Lemmata~5.3 and~5.4 can be proved using our Corollary~\ref{double skewering} as in \cite{CS}, with the additional requirement that all input and output halfspaces be thick. 

Lemma~5.2 of \cite{CS} requires more care. The rest of the proof of Theorem~\ref{strong separation} will therefore be devoted to obtaining the following version of it:
\begin{center}
\emph{``If $\mf{h},\mf{k}$ are thick transverse halfspaces, one of the four sectors determined by $\mf{h}$ and $\mf{k}$ contains a thick halfspace.''}
\end{center}
Let $\mc{H}$ be the set of thick halfspaces that are not transverse to $\mf{h}$ and $\mc{K}$ the set of thick halfspaces that are not transverse to $\mf{k}$. As in \cite{CS}, we can assume that every halfspace in $\mc{H}$ is transverse to every halfspace in $\mc{K}$. Let $\mc{H}'$ be the collection of thick halfspaces that either contain or are contained in some halfspace of $\mc{H}$; we define $\mc{K}'$ similarly. 

Observe that $\mf{a}\in\mscr{H}$ lies in $\mc{H}'$ if and only if there exist $\mf{b},\mf{b}'\in\mc{H}$ such that $\mf{b}\cu\mf{a}\cu\mf{b}'$; again, this is proved as in \cite{CS}. Thus, halfspaces in $(\mscr{H}\setminus\mscr{H}^{\times})\setminus\mc{H}'$ must be transverse to all halfspaces in $\mc{H}'$. We conclude that we have a $*$-invariant partition
\[\mscr{H}=\mc{H}'\sqcup\left(\mscr{H}\setminus(\mc{H}'\cup\mscr{H}^{\times})\right)\sqcup\mscr{H}^{\times},\]
where the first two pieces are transverse and the third is null. Since $\mf{h}\in\mc{H}'$ and $\mf{k}\in\mscr{H}\setminus(\mc{H}'\cup\mscr{H}^{\times})$ this partition is nontrivial. Finally, observe that $\mc{H}'$ is inseparable and, thus, measurable. Proposition~\ref{products} now violates the irreducibility of $X$.
\end{proof}

\section{Facing triples.}\label{facing triples}

Let $X$ be a complete median space of finite rank $r$, let $\G$ be a group and let $\G\acts X$ be an isometric action without wall inversions. 

In this section we study certain tree-like behaviours displayed by all median spaces that admit a Roller nonelementary action. These will allow us to construct nonabelian free subgroups of their isometry groups, proving Theorem~E.

We say that the median space $X$ is \emph{lineal} with endpoints $\xi,\eta\in\overline X$ if $X\cu I(\xi,\eta)$. 

\begin{lem}\label{intervals vs elementarity}
Every action on a lineal median space is Roller elementary.
\end{lem}
\begin{proof}
The elements of $\mscr{F}:=\left\{\{\xi,\eta\}\cu\overline X\mid X\cu I(\xi,\eta)\right\}\neq\emptyset$ are permuted by each isometry of $X$. If $\{\xi_1,\eta_1\}$, $\{\xi_2,\eta_2\}$ are distinct elements of $\mscr{F}$, the sets $\mscr{H}(\xi_1,\eta_2|\eta_1,\xi_2)$ and $\mscr{H}(\xi_1,\xi_2|\eta_1,\eta_2)$ are transverse and their union is $\mscr{H}(\xi_1|\eta_1)$, which contains a side of every wall of $X$. By Proposition~\ref{products}, $X$ splits as a product $X_1\x X_2$. Thus, if $X$ is irreducible, we have $\#\mscr{F}=1$ and an index-two subgroup of $\G$ fixes two points of $\overline X$.

In general, let $X=X_1\x ... \x X_k$ be the decomposition of $X$ into irreducible factors and $\G$ a group of isometries of $X$. By Proposition~\ref{isometries of products}, a finite-index subgroup $\G_0\leq\G$ leaves this decomposition invariant. Since $\overline X=\overline{X_1}\x ... \x\overline{X_k}$ by Lemma~\ref{Roller for products}, if $X$ is lineal so is each $X_i$. The previous discussion shows that a finite-index subgroup of $\G_0$ fixes points $\xi_i\in\overline{X_i}$, for all $i$; in particular, it fixes the point $(\xi_1,...,\xi_k)\in\overline X$, hence $\G\acts X$ is Roller elementary.
\end{proof}

Halfspaces $\mf{h}_1,\mf{h}_2,\mf{h}_3$ are said to form a \emph{facing triple} if they are pairwise disjoint; if each $\mf{h}_i$ is thick, we speak of a thick facing triple. If $X$ is lineal, $\mscr{H}$ does not contain facing triples. On the other hand, we have the following result; compare with Corollary~2.34 in \cite{CFI} and Theorem~7.2 in \cite{CS}.

\begin{prop}\label{facing 3-ples exist}
\begin{enumerate}
\item If $\G\acts X$ is Roller nonelementary, there exists a thick facing triple.
\item If $X$ is irreducible and $\G\acts X$ is Roller nonelementary and Roller minimal, every thick halfspace is part of a thick facing triple.
\end{enumerate}
\end{prop}
\begin{proof}
We prove part~$(1)$ by induction on the rank; the rank-zero case is trivial. In general, let $C\cu Z$ be as provided by Corollary~\ref{Roller elementary vs strongly so 2}. If ${Z\cu\partial X}$, we have $\text{rank}(C)\leq\text{rank}(X)-1$ by Proposition~\ref{components}; in this case, we conclude by the inductive hypothesis and Proposition~\ref{halfspaces of components}. Otherwise, we have $C\cu X$; let $C=C_1\x ...\x C_k$ be its decomposition into irreducible factors. By Proposition~\ref{isometries of products}, a finite-index subgroup $\G_0\leq\G$ preserves this decomposition and, since $\overline C=\overline{C_1}\x...\x\overline{C_k}$, there exists $i\leq k$ such that $\G_0\acts C_i$ is Roller nonelementary. If $k\geq 2$, we have $\text{rank}(C_i)\leq\text{rank}(X)-1$ and we conclude again by the inductive hypothesis. If $C$ is irreducible, Corollary~\ref{double skewering} and Theorem~\ref{strong separation} provide $\mf{h}\in\mscr{H}(C)\setminus\mscr{H}^{\x}(C)$ and $g\in\G$ such that $g\mf{h}$ and $\mf{h}^*$ are strongly separated. Since $\overline C$ is compact, there exists a point $\xi\in\overline C$ that lies in the closure of every $g^n\wt{\mf{h}}$, $n\in\Z$; similarly, we can find $\eta\in\overline C$ lying in the closure of every $g^n\wt{\mf{h}}^*$, $n\in\Z$.

By Lemma~\ref{intervals vs elementarity}, there exists $x\in C$ with $m:=m(x,\xi,\eta)\neq x$. Picking $\mf{j}\in\mscr{H}(m|x)\setminus\mscr{H}^{\times}(C)$, neither $\xi$ nor $\eta$ can lie in $\wt{\mf{j}}$. Since $g^n\mf{h}$ and $g^m\mf{h}^*$ are strongly separated for $n<m$, the halfspace $\mf{j}$ can be transverse to $g^n\mf{h}$ for at most one $n\in\Z$. Hence $j\cu g^{n+2}\mf{h}\setminus g^n\mf{h}$ for some $n\in\Z$ and $\mf{j}$, $g^n\mf{h}$, $g^{n+2}\mf{h}^*$ form a thick facing triple in $\mscr{H}(C)\cu\mscr{H}$.

We now prove part~$(2)$. Let $\mscr{T}\cu\mscr{H}\setminus\mscr{H}^{\times}$ be the set of halfspaces that are part of a thick facing triple; it is nonempty by part~$(1)$. It is also inseparable, hence measurable. Observe moreover that ${\mscr{T}=\mscr{T}^*}$ by part~$(1)$ of Corollary~\ref{double skewering}; indeed, if $\mf{k},\mf{k}',\mf{k}''$ form a facing triple and $g\mf{k}^*\cu\mf{k}$, the halfspaces $\mf{k}^*,g\mf{k}',g\mf{k}''$ also form a facing triple. If $\mf{h}\in\mscr{H}\setminus\mscr{H}^{\times}$ and ${\mf{k}\in\mscr{T}}$ are not transverse, then $\mf{h}\in\mscr{T}$; indeed, up to replacing $\mf{h}$ and $\mf{k}$ with their complements, we can assume that $\mf{h}\cu\mf{k}$. Thus, if $\mscr{T}\neq\mscr{H}\setminus\mscr{H}^{\times}$, Proposition~\ref{products} yields a contradiction. 
\end{proof}

We will refer to facing triples of pairwise strongly separated halfspaces as \emph{strongly separated triples}. Strongly separated $n$-tuples are defined similarly for all $n\geq 4$.

\begin{lem}\label{strongly separated n-tuples}
Suppose $\G\acts X$ is Roller nonelementary and Roller minimal. If $X$ is irreducible, for every $n\geq 3$ every thick halfspace is part of a strongly separated, thick $n$-tuple.
\end{lem}
\begin{proof}
Each thick halfspace $\mf{h}$ is part of a thick facing triple $\mf{h},\mf{h}_1,\mf{h}_2$ by Proposition~\ref{facing 3-ples exist}. Theorem~\ref{strong separation} and Corollary~\ref{double skewering} yield thick halfspaces $\mf{k}_i\cu\mf{h}_i$ such that $\mf{k}_i$ and $\mf{h}_i^*$ are strongly separated; we obtain a strongly separated triple $\mf{h},\mf{k}_1,\mf{k}_2$. We conclude by showing that, if $n\geq 2$, any $(n+1)$-tuple $\mf{h},\mf{h}_1,...,\mf{h}_n$ can be upgraded to a $(n+2)$-tuple. By Corollary~\ref{double skewering}, there exists $g\in\G$ such that $g\mf{h}^*\cu\mf{h}_n$. It is immediate to check that $\mf{h},\mf{h}_1,...,\mf{h}_{n-1},g\mf{h}_1,g\mf{h}_2$ is a strongly separated, thick $(n+2)$-tuple.
\end{proof}

The following result allows us to construct free subgroups of $\G$; in the case of ${\rm CAT}(0)$ cube complexes, compare with Theorem~6 in \cite{Delzant-Py} and Theorem~F in \cite{CS}. Note that only part~$(1)$ is needed to prove the Tits alternative; we will use part~$(2)$ in \cite{Fioravanti3} to characterise Roller elementarity in terms of the vanishing of a certain cohomology class.

\begin{prop}\label{free subgroups acting freely}
\begin{enumerate}
\item Suppose that $\G\acts X$ is Roller minimal and $\mf{h}$ is part of a thick facing triple. A subgroup of $\G$ is free on two generators and has trivial intersection with the stabiliser of the wall $\{\mf{h},\mf{h}^*\}$.
\item Suppose that $\G\acts X$ is Roller minimal and Roller elementary; assume moreover that $X$ is irreducible. There exists a free subgroup $H=\langle a,b\rangle\leq\G$ and a measurable partition $\mscr{H}=\bigsqcup_{h\in H}\mscr{H}_{h}$, where $\mscr{H}_h^*=\mscr{H}_h$ and $g\mscr{H}_h=\mscr{H}_{gh}$, for all $g,h\in H$. In particular, $H$ acts freely on $\mscr{W}$. 
\end{enumerate}
\end{prop}
\begin{proof}
Any thick facing triple including $\mf{h}$ can be upgraded to a thick facing $4$-tuple $\mf{h},\mf{h}',\mf{k},\mf{k}'$ as in the proof of Lemma~\ref{strongly separated n-tuples}. Corollary~\ref{double skewering} gives us $a,b\in\G$ with $a\mf{h}^*\subsetneq\mf{h}'$ and $b\mf{k}^*\subsetneq\mf{k}'$; in particular, $\mf{h}, a\mf{h}^*, \mf{k}, b\mf{k}^*$ form a facing $4$-tuple. Set $\mf{w}:=\{\mf{h},\mf{h}^*\}\in\mscr{W}$ and $\Om:=\mf{h}^*\cap a\mf{h}\cap\mf{k}^*\cap b\mf{k}\cu X$. 

{\bf Claim.} \emph{For every nontrivial, reduced word $u$ in $a$ and $b$,
\begin{itemize}
\item if $u=au'$, we have $u\Om\cu a\mf{h}^*$;
\item if $u=a^{-1}u'$, we have $u\Om\cu\mf{h}$;
\item if $u=bu'$, we have $u\Om\cu b\mf{k}^*$;
\item if $u=b^{-1}u'$, we have $u\Om\cu\mf{k} $.
\end{itemize}}
\begin{proof}[Proof of claim]
We proceed by induction on the length $|u|$ of the word $u$. If $|u|=1$, the statement is obvious. In the inductive step, we can assume that the statement holds for $u'$. Moreover, in cases (a) and (b), the word $u'$ cannot begin with $a^{-1}$ or $a$, respectively, and in cases (c) and (d), the word $u'$ cannot begin with $b^{-1}$ or $b$, respectively. We conclude by observing that 
\[a\left(a\mf{h}^*\cup b\mf{k}^*\cup \mf{k}\right)\cu a\mf{h}^*,\]
\[a^{-1}\left(\mf{h}\cup b\mf{k}^*\cup \mf{k}\right)\cu \mf{h},\]
\[b\left(a\mf{h}^*\cup\mf{h}\cup b\mf{k}^*\right)\cu b\mf{k}^* ,\]
\[b^{-1}\left(a\mf{h}^*\cup\mf{h}\cup\mf{k}\right)\cu \mf{k} .\]
\end{proof}
As a consequence of the claim, if there existed a nontrivial, reduced word $u$ in $a$ and $b$ with $u\mf{w}=\mf{w}$, then we would have $u=a^{-1}u'$ and ${u\mf{h}^{*}=\mf{h}}$. However, this would violate the assumption that $\G$ act without wall inversions. This concludes the proof of part~$(1)$.
\begin{figure}
\centering
\includegraphics[height=2.5in]{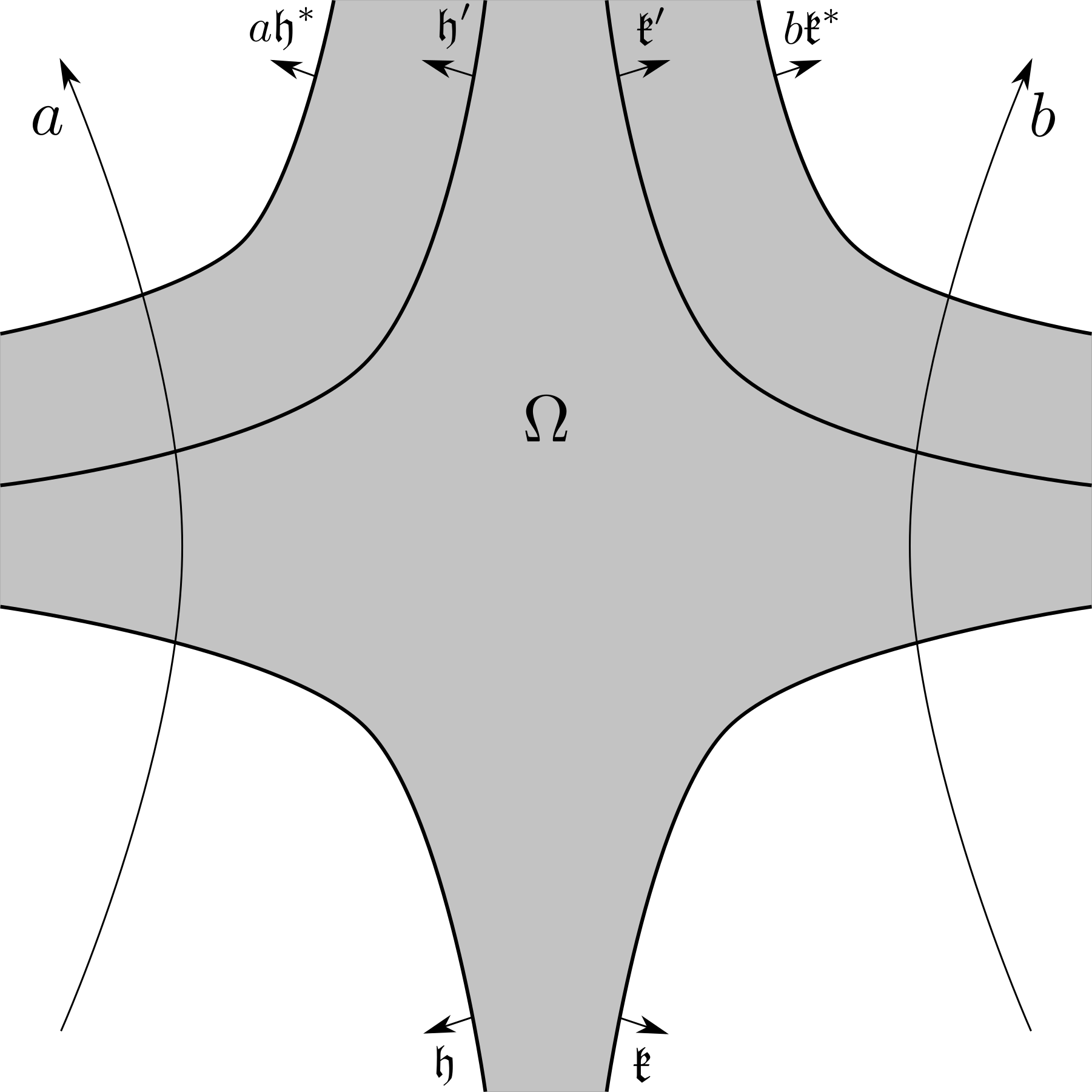}
\caption{}
\label{constructing a free group}
\end{figure}

Under the hypotheses of part~$(2)$, Lemma~\ref{strongly separated n-tuples} provides a strongly separated $4$-tuple $\mf{h},\mf{h}',\mf{k},\mf{k}'$ and we can again construct $a,b\in\G$ and $\Om\cu X$ as above; in particular, $\mf{h}, a\mf{h}^*, \mf{k}, b\mf{k}^*$ form a strongly separated $4$-tuple. Let $\mscr{H}_1$ the set of halfspaces $\mf{j}\in\mscr{H}$ satisfying one of the following two conditions:
\begin{itemize}
\item $\mf{j}$ is transverse to $\mf{h}$ or $\mf{k}$, or $\mf{j}\in\{\mf{h},\mf{h}^*,\mf{k},\mf{k}^*\}$;
\item the wall $\{\mf{j},\mf{j}^*\}$ is contained in $\Om$ and $\mf{j}\not\in H\cdot\{\mf{h},\mf{h}^*,\mf{k},\mf{k}^*\}$.
\end{itemize}
We set $\mscr{H}_h:=h\cdot\mscr{H}_1$ for every $h\in H$. From part~$(1)$, we have that ${h\Om\cap h'\Om=\emptyset}$ whenever $h\neq h'$. Moreover, $H$ is free on $a$ and $b$ and intersects trivially the stabilisers of the walls $\{\mf{h},\mf{h}^*\}$ and $\{\mf{k},\mf{k}^*\}$. Any two distinct walls in $H\cdot\{\mf{h},\mf{h}^*\}\cup H\cdot\{\mf{k},\mf{k}^*\}$ correspond to a strongly separated pair of halfspaces, hence $\mscr{H}_h\cap\mscr{H}_{h'}=\emptyset$, whenever $h\neq h'$

We are left to prove that every $\mf{j}\in\mscr{H}$ lies in some $\mscr{H}_h$. The intersection graph of the regions $h\overline\Om$, $h\in H$, is canonically isomorphic to the Cayley graph of $\left(H,\{a,b\}\right)$. In particular, DCC ultrafilters (in the sense of \cite{Sageev-notes}) on the pocset $H\cdot\{\mf{h},\mf{h}^*,\mf{k},\mf{k}^*\}$ are in one-to-one correspondence with regions $h\Om$. If $\mf{j}$ is not transverse to any element of $H\cdot\{\mf{h},\mf{h}^*,\mf{k},\mf{k}^*\}$, then the set of halfspaces in $H\cdot\{\mf{h},\mf{h}^*,\mf{k},\mf{k}^*\}$ that contain either $\mf{j}$ or $\mf{j}^*$ is a DCC ultrafilter and, by the previous discussion, it corresponds to some $h\Om$, $h\in H$. We conclude that the wall $\{\mf{j},\mf{j}^*\}$ is contained in $h\Om$.
\end{proof}

\begin{cor}\label{pre Tits alternative}
Either $\G$ has a nonabelian free subgroup or ${\G\acts X}$ is Roller elementary. In the latter case, $\G$ is virtually locally-elliptic-by-abelian.
\end{cor}
\begin{proof}
If $\G\acts X$ is Roller nonelementary, Corollary~\ref{Roller elementary vs strongly so 2} yields $C\cu\overline X$ such that $\G\acts C$ is Roller minimal and Roller nonelementary. Part~$(1)$ of Proposition~\ref{facing 3-ples exist} and part~$(1)$ of Proposition~\ref{free subgroups acting freely} yield a nonabelian free subgroup.

If, instead, $\G\acts X$ is Roller elementary, $\G$ has a finite-index subgroup fixing a point $\xi\in\overline X$. If $\xi\in X$, the group $\G$ is elliptic. Otherwise $\xi\in\partial X$ and Theorem~\ref{stabiliser of xi} yields a further finite-index subgroup $\G_0\leq\G$ that is locally-elliptic-by-abelian.
\end{proof}

One can remove the assumption that $\G$ act without wall inversions by passing to the barycentric subdivision $X'$ and appealing to Lemma~\ref{RNE for X'} and part~$(3)$ of Proposition~\ref{properties of X'}. This yields Theorem~E.

\bibliography{mybib}

\newcommand{\etalchar}[1]{$^{#1}$}
\begin{thebibliography}{BCG{\etalchar{+}}09}

\bibitem[BC12]{Behrstock-Charney}
Jason Behrstock and Ruth Charney.
\newblock Divergence and quasimorphisms of right-angled {A}rtin groups.
\newblock {\em Math. Ann.}, 352(2):339--356, 2012.

\bibitem[BCG{\etalchar{+}}09]{BCGNW}
Jacek Brodzki, Sarah~J. Campbell, Erik~P. Guentner, Graham~A. Niblo, and
  Nick~J. Wright.
\newblock Property {A} and {$\rm CAT(0)$} cube complexes.
\newblock {\em J. Funct. Anal.}, 256(5):1408--1431, 2009.

\bibitem[BDS11a]{Behrstock-Drutu-Sapir2}
Jason Behrstock, Cornelia Dru\c{t}u, and Mark Sapir.
\newblock Addendum: {M}edian structures on asymptotic cones and homomorphisms
  into mapping class groups [mr2783135].
\newblock {\em Proc. Lond. Math. Soc. (3)}, 102(3):555--562, 2011.

\bibitem[BDS11b]{Behrstock-Drutu-Sapir}
Jason Behrstock, Cornelia Dru\c{t}u, and Mark Sapir.
\newblock Median structures on asymptotic cones and homomorphisms into mapping
  class groups.
\newblock {\em Proc. Lond. Math. Soc. (3)}, 102(3):503--554, 2011.

\bibitem[BH99]{BH}
Martin~R. Bridson and Andr\'e Haefliger.
\newblock {\em Metric spaces of non-positive curvature}, volume 319 of {\em
  Grundlehren der Mathematischen Wissenschaften [Fundamental Principles of
  Mathematical Sciences]}.
\newblock Springer-Verlag, Berlin, 1999.

\bibitem[BH16]{Behrstock-Hagen}
Jason Behrstock and Mark~F. Hagen.
\newblock Cubulated groups: thickness, relative hyperbolicity, and simplicial
  boundaries.
\newblock {\em Groups Geom. Dyn.}, 10(2):649--707, 2016.

\bibitem[BHS15]{HHS2}
Jason Behrstock, Mark~F. Hagen, and Alessandro Sisto.
\newblock Hierarchically hyperbolic spaces {II}: combination theorems and the
  distance formula.
\newblock {\em arXiv:1509.00632v3}, 2015.

\bibitem[BHS17]{HHS}
Jason Behrstock, Mark~F. Hagen, and Alessandro Sisto.
\newblock Hierarchically hyperbolic spaces, {I}: {C}urve complexes for cubical
  groups.
\newblock {\em Geom. Topol.}, 21(3):1731--1804, 2017.

\bibitem[BM11]{Behrstock-Minsky}
Jason~A. Behrstock and Yair~N. Minsky.
\newblock Centroids and the rapid decay property in mapping class groups.
\newblock {\em J. Lond. Math. Soc. (2)}, 84(3):765--784, 2011.

\bibitem[Bow13]{Bow1}
Brian~H. Bowditch.
\newblock Coarse median spaces and groups.
\newblock {\em Pacific J. Math.}, 261(1):53--93, 2013.

\bibitem[Bow15]{Bow5}
Brian~H. Bowditch.
\newblock Median and injective metric spaces.
\newblock Preprint, 2015.

\bibitem[Bow16]{Bow4}
Brian~H. Bowditch.
\newblock Some properties of median metric spaces.
\newblock {\em Groups Geom. Dyn.}, 10(1):279--317, 2016.

\bibitem[Bow18]{Bow3}
Brian~H. Bowditch.
\newblock Large-scale rigidity properties of the mapping class groups.
\newblock {\em Pacific J. Math.}, 293(1):1--73, 2018.

\bibitem[BW12]{Bergeron-Wise}
Nicolas Bergeron and Daniel~T. Wise.
\newblock A boundary criterion for cubulation.
\newblock {\em Amer. J. Math.}, 134(3):843--859, 2012.

\bibitem[CDH10]{CDH}
Indira Chatterji, Cornelia Dru\c{t}u, and Fr\'ed\'eric Haglund.
\newblock Kazhdan and {H}aagerup properties from the median viewpoint.
\newblock {\em Adv. Math.}, 225(2):882--921, 2010.

\bibitem[CFI16]{CFI}
Indira Chatterji, Talia Fern\'os, and Alessandra Iozzi.
\newblock The median class and superrigidity of actions on {$\rm CAT(0)$} cube
  complexes.
\newblock {\em J. Topol.}, 9(2):349--400, 2016.
\newblock With an appendix by Pierre-Emmanuel Caprace.

\bibitem[CL10]{CL}
Pierre-Emmanuel Caprace and Alexander Lytchak.
\newblock At infinity of finite-dimensional {${\rm CAT}(0)$} spaces.
\newblock {\em Math. Ann.}, 346(1):1--21, 2010.

\bibitem[CMV04]{Cherix-Martin-Valette}
Pierre-Alain Cherix, Florian Martin, and Alain Valette.
\newblock Spaces with measured walls, the {H}aagerup property and property
  ({T}).
\newblock {\em Ergodic Theory Dynam. Systems}, 24(6):1895--1908, 2004.

\bibitem[Cor13]{Cor}
Yves Cornulier.
\newblock Group actions with commensurated subsets, wallings and cubings.
\newblock {\em arXiv:1302.5982v2}, 2013.

\bibitem[CRK15]{Casals-Kazachkov}
Montserrat Casals-Ruiz and Ilya Kazachkov.
\newblock Limit groups over partially commutative groups and group actions on
  real cubings.
\newblock {\em Geom. Topol.}, 19(2):725--852, 2015.

\bibitem[CS11]{CS}
Pierre-Emmanuel Caprace and Michah Sageev.
\newblock Rank rigidity for {${\rm CAT}(0)$} cube complexes.
\newblock {\em Geom. Funct. Anal.}, 21(4):851--891, 2011.

\bibitem[Dil50]{Dilworth}
Robert~P. Dilworth.
\newblock A decomposition theorem for partially ordered sets.
\newblock {\em Ann. of Math. (2)}, 51:161--166, 1950.

\bibitem[DK17]{DK}
Cornelia Drutu and Michael Kapovich.
\newblock Geometric group theory.
\newblock In {\em AMS series ``Colloquium Publications''}, 2017.

\bibitem[DP16]{Delzant-Py}
Thomas Delzant and Pierre Py.
\newblock Cubulable {K}\"ahler groups.
\newblock {\em arXiv:1609.08474v1}, 2016.

\bibitem[DS00]{Dunwoody-Swenson}
Martin~J. Dunwoody and Eric~L. Swenson.
\newblock The algebraic torus theorem.
\newblock {\em Invent. Math.}, 140(3):605--637, 2000.

\bibitem[Fer15]{Fernos}
Talia Fern\'os.
\newblock The {F}urstenberg-{P}oisson boundary and {${\rm CAT}(0)$} cube
  complexes.
\newblock {\em arXiv:1507.05511v1}, 2015.

\bibitem[Fio17a]{Fioravanti1}
Elia Fioravanti.
\newblock Roller boundaries for median spaces and algebras.
\newblock {\em arXiv:1708.01005v2}, 2017.

\bibitem[Fio17b]{Fioravanti3}
Elia Fioravanti.
\newblock Superrigidity of actions on finite rank median spaces.
\newblock {\em arXiv:1711.07737v1}, 2017.

\bibitem[FL08]{Foertsch-Lytchak}
Thomas Foertsch and Alexander Lytchak.
\newblock The de {R}ham decomposition theorem for metric spaces.
\newblock {\em Geom. Funct. Anal.}, 18(1):120--143, 2008.

\bibitem[Gen18]{GenevoisV}
Anthony Genevois.
\newblock Hyperbolic and cubical rigidities of {T}hompson's group {V}.
\newblock {\em arXiv:1804.01791v1}, 2018.

\bibitem[Ger97]{Gerasimov}
Victor~N. Gerasimov.
\newblock Semi-splittings of groups and actions on cubings.
\newblock In {\em Algebra, geometry, analysis and mathematical physics
  ({R}ussian) ({N}ovosibirsk, 1996)}, pages 91--109, 190. Izdat. Ross. Akad.
  Nauk Sib. Otd. Inst. Mat., Novosibirsk, 1997.

\bibitem[Gui05]{Guirardel}
Vincent Guirardel.
\newblock C\oe ur et nombre d'intersection pour les actions de groupes sur les
  arbres.
\newblock {\em Ann. Sci. \'Ecole Norm. Sup. (4)}, 38(6):847--888, 2005.

\bibitem[Hag13]{Hagen}
Mark~F. Hagen.
\newblock The simplicial boundary of a {${\rm CAT}(0)$} cube complex.
\newblock {\em Algebr. Geom. Topol.}, 13(3):1299--1367, 2013.

\bibitem[Hag18]{corrigendum}
Mark~F. Hagen.
\newblock Corrigendum to the article {T}he simplicial boundary of a {${\rm
  CAT}(0)$} cube complex.
\newblock {\em Algebr. Geom. Topol.}, 18(2):1251--1256, 2018.

\bibitem[HW]{Hagen-Wilton}
Mark~F. Hagen and Henry Wilton.
\newblock Guirardel cores for cube complexes.
\newblock In preparation.

\bibitem[Kat18]{Kato}
Motoko Kato.
\newblock On groups whose actions on finite-dimensional {$\rm CAT(0)$} spaces
  have global fixed points.
\newblock {\em arXiv:1804.10506v1}, 2018.

\bibitem[KS16a]{Kar-Sageev}
Aditi Kar and Michah Sageev.
\newblock Ping pong on {$\rm CAT(0)$} cube complexes.
\newblock {\em Comment. Math. Helv.}, 91(3):543--561, 2016.

\bibitem[KS16b]{Kar-Sageev2}
Aditi Kar and Michah Sageev.
\newblock Uniform exponential growth for {$\rm CAT(0)$} square complexes.
\newblock {\em arXiv:1607.00052v2}, 2016.

\bibitem[Min16]{Minasyan}
Ashot Minasyan.
\newblock New examples of groups acting on real trees.
\newblock {\em J. Topol.}, 9(1):192--214, 2016.

\bibitem[MS91]{Morgan-Shalen}
John~W. Morgan and Peter~B. Shalen.
\newblock Free actions of surface groups on {${\bf R}$}-trees.
\newblock {\em Topology}, 30(2):143--154, 1991.

\bibitem[Nic08]{Nica-thesis}
Bogdan Nica.
\newblock Group actions on median spaces.
\newblock {\em arXiv:0809.4099v1}, 2008.

\bibitem[NR98]{Niblo-Roller}
Graham~A. Niblo and Martin~A. Roller.
\newblock Groups acting on cubes and {K}azhdan's property ({T}).
\newblock {\em Proc. Amer. Math. Soc.}, 126(3):693--699, 1998.

\bibitem[NR03]{Niblo-Reeves}
Graham~A. Niblo and Lawrence~D. Reeves.
\newblock Coxeter groups act on {${\rm CAT}(0)$} cube complexes.
\newblock {\em J. Group Theory}, 6(3):399--413, 2003.

\bibitem[NS13]{Nevo-Sageev}
Amos Nevo and Michah Sageev.
\newblock The {P}oisson boundary of {${\rm CAT}(0)$} cube complex groups.
\newblock {\em Groups Geom. Dyn.}, 7(3):653--695, 2013.

\bibitem[OW11]{Ollivier-Wise}
Yann Ollivier and Daniel~T. Wise.
\newblock Cubulating random groups at density less than {$1/6$}.
\newblock {\em Trans. Amer. Math. Soc.}, 363(9):4701--4733, 2011.

\bibitem[Ram30]{Ramsey}
Frank~P. Ramsey.
\newblock On a {P}roblem of {F}ormal {L}ogic.
\newblock {\em Proc. London Math. Soc.}, S2-30:264--286, 1930.

\bibitem[Rol98]{Roller}
Martin~A. Roller.
\newblock Poc sets, median algebras and group actions. {A}n extended study of
  {D}unwoody's construction and {S}ageev's theorem.
\newblock preprint, University of Southampton, 1998.

\bibitem[Sag95]{Sageev}
Michah Sageev.
\newblock Ends of group pairs and non-positively curved cube complexes.
\newblock {\em Proc. London Math. Soc. (3)}, 71(3):585--617, 1995.

\bibitem[Sag14]{Sageev-notes}
Michah Sageev.
\newblock {$\rm CAT(0)$} cube complexes and groups.
\newblock In {\em Geometric group theory}, volume~21 of {\em IAS/Park City
  Math. Ser.}, pages 7--54. Amer. Math. Soc., Providence, RI, 2014.

\bibitem[SW05]{Sageev-Wise}
Michah Sageev and Daniel~T. Wise.
\newblock The {T}its alternative for {${\rm CAT}(0)$} cubical complexes.
\newblock {\em Bull. London Math. Soc.}, 37(5):706--710, 2005.

\bibitem[Wri12]{Wright}
Nick Wright.
\newblock Finite asymptotic dimension for {${\rm CAT}(0)$} cube complexes.
\newblock {\em Geom. Topol.}, 16(1):527--554, 2012.

\bibitem[Zei16]{Zeidler}
Rudolf Zeidler.
\newblock Coarse median structures and homomorphisms from {K}azhdan groups.
\newblock {\em Geom. Dedicata}, 180:49--68, 2016.

\end{thebibliography}
\bibliographystyle{alpha}

\end{document}